\documentclass[11pt]{article}

\usepackage[latin1]{inputenc}
\usepackage{t1enc}
\usepackage{amsmath, amssymb, amsfonts, amsthm, amsopn,epsfig}
\usepackage[none]{hyphenat}
\usepackage{tikz}
\usepackage{float}
\usepackage{graphicx}
\usepackage{caption}
\usepackage[toc,page]{appendix}
\usepackage{epstopdf}
\usepackage{etoolbox}
\usepackage[colorlinks=true]{hyperref}
\usepackage{dsfont}
\hypersetup{
	colorlinks=true,
	linkcolor=blue,
	filecolor=magenta,
	urlcolor=cyan,
	citecolor=blue
}
\usepackage{cleveref}
\usepackage[top=22mm, bottom=22mm, left=22mm, right=22mm]{geometry}
\usepackage{comment}

\theoremstyle{plain}
\newtheorem{theorem}{Theorem}[section]

\newtheorem{claim}[theorem]{Claim}
\newtheorem{lemma}[theorem]{Lemma}
\newtheorem{conjecture}[theorem]{Conjecture}
\newtheorem{problem}[theorem]{Problem}
\newtheorem{proposition}[theorem]{Proposition}

\theoremstyle{definition}

\DeclareMathOperator{\disc}{disc}

\DeclareMathOperator{\rank}{rank}

\newcommand{\cA}{\mathcal{A}}
\newcommand{\cB}{\mathcal{B}}

\newcommand{\cF}{\mathcal{F}}

\newcommand{\cR}{\mathcal{R}}
\newcommand{\cS}{\mathcal{S}}

\newcommand{\bF}{\mathbb{F}}
\newcommand{\bR}{\mathbb{R}}
\newcommand{\bZ}{\mathbb{Z}}

\newcommand{\E}{\mathbb{E}}
\newcommand{\Pb}{\mathbb{P}}

\newcommand{\eps}{\varepsilon}

\newcommand{\hide}[1]{}

\title{
\vspace{-0.8cm}
Disjoint pairs in set systems and combinatorics of low rank matrices}
\author{
Zach Hunter \thanks{Department of Mathematics, ETH Z\"urich, Switzerland. Email: {\tt \{zach.hunter, aleksa.milojevic, benjamin.sudakov\}@math.ethz.ch}. Research supported in part by SNSF grant 200021-228014.}
\and Aleksa Milojevi\'c\footnotemark[1] \and Benny Sudakov \footnotemark[1]
\and 
Istv\'an Tomon\thanks{Ume\r{a} University, \emph{e-mail}: \textbf{istvantomon@gmail.com}, Research supported in part by the Swedish Research Council grant VR 2023-03375.}}
\date{}

\begin{document}
	\sloppy 
 \maketitle

\begin{abstract}

We study and solve several problems in two closely related settings: set families in $2^{[n]}$ with many disjoint pairs of sets and low rank matrices with many zero entries.
\begin{itemize}
	\item More than 40 years ago, Daykin and Erd\H{o}s asked for the maximum number of disjoint pairs of sets in a family $\cF\subseteq 2^{[n]}$ of size $2^{(1/2+\delta)n}$ and conjectured it contains at most $o(|\mathcal{F}|^2)$ such pairs. This was proven by Alon and Frankl in 1985. In this paper we completely resolve this problem, proving an optimal dependence of the number of disjoint pairs on the size of family $\cF$. 
 We also prove the natural variant of the  Daykin-Erd\H{o}s conjecture in which disjoint pairs are replaced by pairs with intersection $\lambda\neq 0$.
 
	\item Motivated by a conjecture of Lovett related to the famous log-rank conjecture, Singer and Sudan asked to show that for two families $\mathcal{A}, \mathcal{B} \subseteq 2^{[n]}$ with a positive constant fraction of pairs $(A,B)\in \mathcal{A}\times\mathcal{B}$ being disjoint, there are $\mathcal{R}\subset \mathcal{A}$ and $\mathcal{S}\subset \mathcal{B}$ such that all pairs $(R, S)\in \mathcal{R}\times \mathcal{S}$ are disjoint, and $|\mathcal{R}|\geq 2^{-O(\sqrt{n})}|\mathcal{A}|$ and $|\mathcal{S}|\geq 2^{-O(\sqrt{n})}|\mathcal{B}|$. We prove this conjecture in a strong quantitative form.

        \item A long-standing problem in coding theory is to determine the largest size of an $r$-cover-free family, which is a family of subsets of $[n]$ such that no set is covered by a union of $r$ other sets. Motivated by this question, Alon, Gilboa and Gueron asked to determine which distribution $\mu$ over $2^{[n]}$ minimizes the probability that $A_0\subseteq A_1\cup\dots\cup A_r$ where $A_0, \dots, A_r$ are random sets drawn independently from $\mu$. Using ideas from the previous bulletpoint, we obtain the tight lower bound $2^{-O(n/r)}$ for this probability.
 
	\item We prove the following generalizations of the best known bounds for the log-rank conjecture. If $M$ is an $n\times n$ non-negative integer matrix of rank $r$ in which the average of the entries is $\eps\leq 1/2$, then $M$ contains an all-zero submatrix of size at least $2^{-O(\sqrt{\eps r})}n$. Unlike the known bounds for the log-rank conjecture, this result is optimal. Moreover, using similar methods, we also prove that any $n\times n$ matrix of rank $r$ with entries from $\{0,\dots,t\}$ contains a constant submatrix of size at least $2^{-O(t\sqrt{r})}n$.
	
\end{itemize}
Our proofs use probabilistic, entropy and discrepancy methods and explore connections to additive combinatorics and coding theory.
\end{abstract}

\section{Introduction}

A central theme of Extremal Set Theory is the study of set systems satisfying certain properties about set intersections. A classical result in the area is the  Erd\H os-Ko-Rado theorem \cite{erdoskorado}, which determines precisely the maximum size of a set system in which no two sets are disjoint.  This naturally motivates further questions about the maximum/minimum number of disjoint pairs of sets in set systems of given size. 
In particular,  Daykin and Erd\H os \cite{Guy83, Banff} proposed the following problem in 1981, see also Alon and Frankl \cite{AF85}.

\begin{problem}[Daykin, Erd\H{o}s]
Determine/estimate the maximum number of pairs of disjoint sets in a set system of size $m$ on a universe of size $n$. 
\end{problem}

This problem  remained mostly open, however, there has been considerable progress. Observe that if $\cF\subset 2^{[n]}$ is the family containing all subsets of $\{1,\dots,n/2\}$ and all subsets of $\{n/2+1,\dots,n\}$, then $|\cF|=2^{n/2+1}$ and $\cF$ contains $\frac{1}{4}|\cF|^2$ (unordered) pairs of disjoint sets. Daykin and Erd\H{o}s conjectured that this is optimal in the weak sense that if a family $\cF\subseteq 2^{[n]}$ contains $\eps |\cF|^2$ pairs of sets $A, B\in \cF$ which are disjoint for some fixed $\eps>0$, then one must have $|\cF|\leq 2^{(1+o(1))n/2}$.

This conjecture  was resolved by Alon and Frankl in 1985 \cite{AF85} by an  elegant probabilistic argument. They show that if $\cF$ is of size $m=2^{(1/2+\delta)n}$, then $\cF$ contains at most $m^{2-\delta^2/2}$ disjoint pairs. However, as they note in their paper, this bound does not appear to be the best possible, and they suggest the following construction may be optimal.

\medskip
\noindent
\textbf{Construction 1.} For every positive integer $d$, let $\cF=\mathcal{A}\cup \mathcal{B}$, where $\mathcal{A}$ is the family of all sets $A$ such that $|A\cap \{1,\dots,n/2\}|\leq d$, and $\mathcal{B}$ is the family of sets $B$ such that $|B\cap \{n/2+1,\dots,n\}|\leq d$.

\medskip

In case $d$ is fixed, the size of $\cF$ is $m = 2\cdot 2^{n/2}\sum_{i=0}^d \binom{n/2}{i}=\Theta_d(2^{n/2} n^d)$ and it contains at least $2^{-2d-2}m^2$ disjoint pairs. On the other hand, taking $d=\Theta(\frac{\delta n}{\log 1/\delta})$, $\cF$ has size $m=2^{(1/2+\delta)n}$ and contains $m^{2-O(\frac{\delta}{\log1/\delta})}$ disjoint pairs. In our first theorem, we show that this construction is optimal asymptotically, and thus we settle the problem of Daykin and Erd\H os in the most interesting range $m\gg 2^{n/2}$.

\begin{theorem}\label{thm:coarse erdos daykin}
There exists $c>0$ such that for every positive integer  $n$ and  $\delta\in (0, 1/2)$, the following holds. Let $\cF\subseteq 2^{[n]}$ be a set family of size $m\geq 2^{(1/2+\delta)n}$. Then there are at most $m^{2-\frac{c\delta}{\log 1/\delta}}$ disjoint pairs in $\cF$.
\end{theorem}

It is often more natural to consider this problem in the bipartite setting. That is, instead of a single family $\cF$, we look at disjoint pairs $(A,B)\in \cA\times \cB$, where $\cA$ and $\cB$ might be different set systems. Considering the families $\cA$ and $\cB$ described in Construction 1, Alon and Frankl \cite{AF85} actually made the following more precise conjecture. If $\cA, \cB\subset 2^{[n]}$ are of size $n^{\omega(1)}2^{n/2}$, then there cannot be $\Omega(|\cA||\cB|)$  disjoint pairs $(A,B)\in \cA\times \cB$.

In 2015, Alon, Das, Glebov and Sudakov \cite{ADGS15} proved this conjecture and showed that for any two families $\cA, \cB\subseteq 2^{[n]}$ with $|\cA||\cB|=\Theta_d(n^{2d}2^n)$, one has at most $2^{-d/150}|\cA||\cB|$ disjoint pairs. However, the problem of determining the maximal density of disjoint pairs between two families $\cA, \cB$ of size $m=\Omega_d(n^d 2^{n/2})$ was not settled. We resolve this problem completely by showing that two families of size $m=\Omega_d(2^{n/2}n^d)$ cannot have more than $(1+o(1))2^{-2d}|\cA||\cB|$ disjoint pairs, which is optimal up to the $o(1)$ term by considering the families $\cA$ and $\cB$ described above.

\begin{theorem}\label{thm:fine-grained daykin erdos}
Let $d>0$ be fixed and let $c_d> 0$ an arbitrary constant depending only on $d$. If $\cA, \cB\subseteq 2^{[n]}$ are set families of size $|\cA||\cB|\geq c_d 2^n n^{2d}$, then the number of disjoint pairs in $\cA\times\cB$ is at most $(1+o(1))2^{-2d}|\cA||\cB|$, where $o(1)\to 0$ as $n\to \infty$.
\end{theorem}

\noindent
Note that this result differs by a factor $2$ from Construction 1, because it is stated in the bipartite setting. However, the non-bipartite statement can easily be reduced to the bipartite one, using Szemer\'edi's regularity lemma. Therefore, one can deduce from the above theorem that families of size $|\cF|=c_d n^{d} 2^{n/2}$  have at most $(1+o(1))2^{-2d-1}\binom{|\cF|}{2}$ disjoint pairs, matching Construction 1. We give the sketch of this reduction following the proof of Theorem~\ref{thm:fine-grained daykin erdos} in Section~\ref{sect:daykin-erdos}.

Having determined the maximum possible density of disjoint pairs quite precisely, it is natural to ask whether these results can be extended to other intersection sizes. Therefore, we propose to study the following problem.

\begin{problem}
Given positive integers $m,n,\lambda$, determine/estimate the maximum number of pairs of sets $(A,B)\in \cF^{2}$ such that $|A\cap B|=\lambda$ in a set system $\cF\subset 2^{[n]}$ of size $m$.
\end{problem}

This problem turns out to be substantially different from the case of disjoint pairs, and the same techniques do not seem to apply anymore. We use tools from additive combinatorics to prove the following, which is the natural generalization of the aforementioned conjecture of Daykin and Erd\H{o}s.

\begin{theorem}\label{thm:daykin erdos nonzero intersection}
Let $c>0$ be a fixed constant and let $n\geq \lambda>0$ be positive integers. Let $\mathcal{A}, \mathcal{B}\subseteq 2^{[n]}$ be families of size $m$ such that at least $cm^2$ pairs $(A, B)\in \mathcal{A}\times \mathcal{B}$ satisfy $|A\cap B|=\lambda$. Then $$m\leq 2^{(1/2+o(1))n}.$$
\end{theorem}

\noindent
Beyond $\lambda=0$, this result is tight for various other $\lambda$. For example, for $\lambda=n/4$, split the ground set $[n]$ into $n/2$ pairs, take $\cA$ to be the family of all sets containing $n/4$ pairs, and take $\cB$ to be the family of all sets containing one element from every pair.

The proof of Theorem~\ref{thm:daykin erdos nonzero intersection} relies fundamentally on a connection with a problem from additive combinatorics, called the Approximate Duality Conjecture. This conjecture was originally proposed by Ben-Sasson and Ron-Zewi  \cite{BSRZ15} as an approach to the log-rank conjecture. We discuss the log-rank conjecture in more detail in the next section. The Approximate Duality Conjecture states if $M$ is an $m\times m$ matrix of rank $n$ over $\mathbb{F}_2$ with at least $(1/2+\delta)m^2$ equal entries, then one can find a constant submatrix of $M$ of size $\exp(-O(\sqrt{n\log \delta^{-1}}))m$ (i.e. a square submatrix with sides of length $\exp(-O(\sqrt{n\log \delta^{-1}}))m$).
While this conjecture remains open, partial progress towards it was made by Ben-Sasson, Lovett and Ron-Zewi \cite{BSLRZ14}. They show the existence of a constant square submatrix of size $\exp(-O(n/\log n))m$, conditionally on the Polynomial Freiman-Ruzsa conjecture. 

The Polynomial Freiman-Ruzsa conjecture was, until recently, a central open problem in additive combinatorics, stating that if the set $A\subseteq \bF_2^n$ satisfies $|A+A|=\big|\{a_1+a_2|a_1, a_2\in A\}\big|\leq K|A|$, then $A$ can be covered by $K^{O(1)}$ translates of a subgroup $H\subseteq \bF_2^n$ of size $|H|\leq |A|$. This conjecture was proven by Gowers, Green, Manners and Tao \cite{GGMT23} (see also \cite{GGMT24}), and therefore the proof of Ben-Sasson, Lovett and Ron-Zewi applies unconditionally. This,  in turn allows us to use it for the proof of Theorem~\ref{thm:daykin erdos nonzero intersection}.

Finally, motivated by both  Theorem \ref{thm:daykin erdos nonzero intersection} and the Approximate Duality Conjecture, we give an elementary proof of stronger bounds for set systems with many even (or odd) intersections.

\begin{theorem}\label{thm:even}
 Let $\delta\in [0,1/2]$, and let $\cA,\cB\subset 2^{[n]}$ such that at least $(1/2+\delta)|\cA||\cB|$ pairs of sets $(A,B)\in \cA\times \cB$ satisfy that $|A\cap B|$ is even. Then $|\cA||\cB|\leq 2^{n}/(4\delta^2)$. The same result remains true if even is replaced by odd.
\end{theorem}

\subsection{Complete bipartite graphs in disjointness graphs}

We now turn our attention to a closely related extremal problem about dense disjointness graphs of set families, namely characterizing the size of the largest biclique they contain. Given two set families $\cA$ and $\cB$ with $N$ disjoint pairs $(A,B)\in \cA\times \cB$,  let $d(\cA,\cB)=\frac{N}{|\cA||\cB|}$ denote the density of disjoint pairs. We are interested in the following problem, proposed by Singer and Sudan \cite{SS22}.

\begin{problem}\label{problem:SS}
Let $\cA, \cB\subseteq 2^{[n]}$ satisfying $d(\cA, \cB)\geq \delta$, then how large families $\cA'\subseteq \cA, \cB'\subseteq \cB$ can we find with the property that every pair of sets $(A,B)\in \cA'\times \cB'$ are disjoint? 
\end{problem}

There are two different motivations for studying this problem. The first one comes from Extremal Graph Theory, where one is often interested in determining the minimal number of edges which are sufficient to guarantee a certain structure in a graph. Zarankiewicz's problem is one of the long-standing open problems in the field asking to find the maximum number of edges in a bipartite graph $G$ with parts of size $m$ and $n$, which does not contain a complete bipartite graph with parts of size $s$ and $t$.  In the past two decades, the study of Zarankiewicz type problems in which the host graph $G$ is assumed to satisfy  certain structural, algebraic, or geometric conditions, gained a lot of attention. We refer the interested reader to the recent survey of Smorodinsky \cite{zara_survey} for the geometric setting, or the manuscript of the authors of this paper \cite{HMST} for an overview of the structural setting. Considering graphs which arise as the disjointness graph of set systems fits well into this area.

On the other hand, Problem \ref{problem:SS} is connected to the problem of finding monochromatic rectangles in low-rank matrices. The main question in this area is the well-known log-rank conjecture of Lov\'asz and Saks \cite{LS88}, which is one of the fundamental open problems in communication complexity. Its combinatorial formulation asks to show that  each binary matrix $M$ of size $n\times n$ and rank $r$ contains an all-zero or all-one submatrix of size at least $2^{-(\log r)^{O(1)}}n$. Such submatrices are called \emph{monochromatic} or \emph{constant}, and they do not have to consist of consecutive rows/columns. While the log-rank conjecture is still outside the reach of known techniques, Lovett \cite{L16} showed  that one can always find a monochromatic submatrix of size $2^{-O(\sqrt{r}\log r)}n$. Recently, this was slightly improved by Sudakov and Tomon \cite{ST24}, who removed the logarithmic factor from the exponent.

In the same paper \cite{L16}, Lovett proposed to study the size of all-zero rectangles in low-rank sparse real matrices, stating the following conjecture.

\begin{conjecture}[\cite{L16}]\label{conj:lovett}
Let $M$ be an $n \times n$ real matrix with ${\rm rank}(M) = r$ such that at most $\eps n^2$ entries of $M$ are not zero, where $\eps\in (0,1/2)$. Then $M$ contains  an all-zero square submatrix of size at least $ n \cdot \exp(-O(\sqrt{\eps r})).$
\end{conjecture}

 Observe that zero patterns of rank $r$ matrices represent incidence structures of points and hyperplanes in $\mathbb{R}^{r-1}$. Indeed, given $m$ points $p_1,\dots,p_m\in \mathbb{R}^{r-1}$ and $n$ hyperplanes $H_1,\dots,H_n$ with equations $H_j=\{\langle x, q_j\rangle=a_j\}$, the matrix $M(i,j)=\langle (p_i,1),(q_j,-a_j)\rangle=\langle p_i, q_j\rangle -a_j$ has rank at most $r$, and $M(i, j)=0$ if and only if $p_i\in H_j$. Furthermore, for every rank $r$ matrix $M$ one can find suitable points and hyperplanes in $\mathbb{R}^{r-1}$ realizing $M$. Thus, Conjecture \ref{conj:lovett} is equivalent to the following conjecture about incidence graphs of points and hyperplanes.

 \begin{conjecture}\label{conj:incidence}
 Let $\eps\in (0,1/2)$, then every incidence graph of $n$ points and $n$ hyperplanes in $\mathbb{R}^{d}$ with at least $(1-\varepsilon) n^2$ edges contains a complete balanced bipartite subgraph of size at least $ n \cdot \exp(-O(\sqrt{\eps d})).$
 \end{conjecture}
  Finding complete bipartite subgraphs in incidence graphs of points and hyperplanes is a celebrated problem of Chazelle \cite{chazelle} from 1993, but most of the research directed towards this problem focused on the case of sparse incidence graphs. In contrast, Conjecture \ref{conj:lovett} is concerned with very dense incidence graphs. Furthermore, Lovett \cite{L16} notes that Conjecture \ref{conj:lovett}, if true, would imply improved bounds for matrix rigidity.
 
Addressing Conjecture \ref{conj:incidence} (and thus Conjecture \ref{conj:lovett} as well), Singer and Sudan \cite{SS22} proved that every incidence graph of $n$ points and $m$ hyperplanes in $\mathbb{R}^d$ with $\delta m n$ edges contains a complete bipartite subgraph with $\Omega(\delta^{2d}mn/d)$ edges, as long as $\delta \geq (n/2)^{-1/d}$. This implies that Conjecture \ref{conj:lovett} holds with the weaker lower bound $n\cdot \exp(-O(r))$. On the other hand, they show that if the conjecture is true, it is sharp. In particular, a construction achieving the bound $ n \cdot \exp(-O(\sqrt{\eps r}))$ can be constructed with the help of set systems. Given a pair of set systems $\cA,\cB \subset 2^{[r]}$, the \emph{intersection matrix} of $(\cA,\cB)$ is the matrix $M$, whose rows and columns are indexed by the elements of $\cA$ and $\cB$, respectively, and $M(A,B)=|A\cap B|$. Note that $\rank(M)\leq r$.

\medskip
\noindent
\textbf{Construction 2.} Let $\cF=\binom{[r]}{k}$ be the family of all $k$-element subsets of $[r]$, where $k=\sqrt{\eps r}$, and let $M_2$ be the $n\times n$ intersection matrix of $(\cF,\cF)$, where $n=\binom{r}{k}$.

\medskip

Then all but $\Theta(\eps n^2)$ entries of $M_2$ are 0, and $M_2$ contains no all-zero submatrix of size $ n \cdot 2^{-\sqrt{\eps r}}$. See Section \ref{sect:construction} for detailed calculations. This motivates the problem of studying disjointness graphs of set systems. Specifically, Singer and Sudan \cite{SS22} proposed the following conjecture. For any two set families $\cA, \cB \subseteq 2^{[n]}$ with $d(\cA, \cB)\geq 1-\eps$, one can find subfamilies $\cR \subseteq\cA, \cS \subseteq \cB$ such that $d(\cR, \cS)=1$ and $|\cR||\cS| \geq |\cA||\cB| \cdot 2^{-O_\eps(\sqrt{n})}$. In the next theorem, we show not only that this is true when $\eps<1/2$, but also find the optimal behavior in the case of sparse disjointness graphs. We say the two families $\cR$ and $\cS$ are \emph{cross-disjoint}, if $d(\cR,\cS)=1$, or in other words, $R$ and $S$ are disjoint for every $R\in\cR$ and $S\in\cS$.

\begin{theorem}\label{thm:singer sudan sparse}
Let $\cA, \cB\subset 2^{[n]}$ such that  $d(\cA,\cB)\geq \delta$  for some $\delta\in (0, 1-\frac{1}{n})$. Then there exist cross-disjoint subfamilies $\cR\subseteq \cA, \cS\subseteq \cB$ which satisfy \[|\cR||\cS|\geq {2^{-O(\sqrt{n\log 1/\delta})}}{|\cA||\cB|}.\]
\end{theorem}

Observe that substituting $\delta=1-\eps$ for some $\eps\in (0,1/2)$ and noting that $\log 1/\delta=\Theta(\eps)$ gives the conjectured bound of Singer and Sudan \cite{SS22}. On the other hand, it is worth pointing out that the analogue of Theorem \ref{thm:singer sudan sparse} for low rank matrices does not hold if $\delta$ is small. More precisely, if $M\in \bR^{n\times n}$ is a matrix of rank $r$ with $\delta n^2$ zero entries, the following construction shows that one cannot always guarantee an all-zero submatrix of size ${2^{-O(\sqrt{r\log 1/\delta})}}n$.

\medskip

\noindent
\textbf{Construction 3.} Let $r$ be an integer divisible by $4$ and let $N$ be the intersection matrix of $(2^{[r]}, 2^{[r]})$. Then, we set $M_3=N-\frac{r}{4}J$, where $J$ is an all-one matrix.
\medskip

Then $M_3$ is a $n\times n$ matrix, with $n=2^r$ and  $\rank(M_3)\leq r+1$. Moreover, $M_3$ has $\Omega(n^2/\sqrt{r})$ zero entries but $M_3$ contains no all-zero $s\times t$ sized submatrix if $st\geq 2^{r}=2^{-r}n^2$. We give detailed proofs of both of these properties in Section \ref{sect:construction}.

\subsection{Cover-free families and connections to coding theory}

The key insight in our proof of Theorem~\ref{thm:singer sudan sparse} is a covering result, which states that for a family $\cA\subseteq 2^{[n]}$ and uniformly random sets $A_0, \dots, A_r\in \cA$, the probability that $A_0\subseteq A_1\cup\dots\cup A_r$ is at least $2^{-O(n/r)}$. This is closely related to a concept in coding theory called $r$-cover-free families and it answers a question which was recently studied by Alon, Gilboa and Gueron \cite{AGG20}.

A set family $\cF\subseteq 2^{[n]}$ is called \textit{$r$-cover-free} if no set of $\cF$ is covered by the union of $r$ others. In other words, we require that there do not exist $r+1$ distinct sets $A_0, \dots, A_r$ for which $A_0\subseteq A_1\cup\dots\cup A_r$. 
This notion was originally introduced in 1964 by Kautz and Singleton \cite{KS64} in the context of coding theory, where they called such families \textit{disjunctive codes}.
However, this notion is also natural combinatorially, witnessed by the fact that it was independently introduced by Erd\H os, Frankl and F\"uredi \cite{EFF85} in 1985. Determining the maximum size of an $r$-cover-free family has been a long-standing open problem, of interest both in combinatorics and in coding theory. The best known upper bounds are due to D'yachkov and Rykov \cite{DR82} who show $|\cF|\leq r^{O(n/r^2)}$. F\"uredi \cite{F96} provides an elegant combinatorial argument which gives the same bound up to the constant in the exponent. However, the best known constructions of $r$-cover-free families only satisfy $|\cF|\geq 2^{\Omega(n/r^2)}$ (\cite{KS64}, \cite{EFF85}). Determining the correct base of the exponent remains an important open problem.

Recently, Alon, Gilboa and Gueron \cite{AGG20} approached this problem from a probabilistic perspective. They asked what probability distribution $\mu$ on $2^{[n]}$ minimizes the probability $\mathbb{P}[A_0\subseteq A_1\cup\dots \cup A_r]$  if $A_0,\dots,A_r$ are sampled independently with respect to $\mu$. 
A natural candidate for a distribution minimizing this probability is the uniform distribution on a maximal $r$-cover-free family $\cF$. Surprisingly, they show that this is not the case.

One of the key observations in Alon, Gilboa and Gueron \cite{AGG20}  is that for any distribution $\mu$, $\Pb[A_0\subseteq A_1\cup\dots\cup A_r]\geq \frac{1}{r^{O(n/r)}}$. This follows from combining the best known upper bound on $r$-cover-free families with a supersaturation result. We strengthen this lower bound, obtaining a tight result up to the constant in the exponent.

\begin{lemma} \label{lemma:covering lemma}
Let $n\geq r\geq 1$ be positive integers and let $\mu$ be a probability distribution on $2^{[n]}$. If $A_0, \dots, A_r$ are randomly and independently drawn elements of $2^{[n]}$ with respect to $\mu$, then
\[\Pb[A_0\subseteq A_1\cup\dots\cup A_r]\geq 2^{-n/r-2}.\]
\end{lemma}

We show that this lower bound is indeed optimal. Let $\mu$ be the $p$-biased distribution on $2^{[n]}$, that is, $\mu(A)=p^{|A|}(1-p)^{n-|A|}$ for every $A\subseteq [n]$. It is easy to calculate that $\Pb[A_0\subseteq A_1\cup\dots\cup A_r]=(1-p(1-p)^r)^n$. This quantity is minimized if $p=\frac{1}{r+1}$, in which case  $1-p(1-p)^r=1-\frac{1+o(1)}{e(r+1)}$ and 
$(1-p(1-p)^r)^n=2^{-\Theta(n/r)}.$

\subsection{Combinatorics of low-rank matrices}\label{sect:intro_matrices}

As already discussed above, pairs of set systems $(\cA,\cB)$ on an $r$ element universe give rise to rank $r$ matrices that record the disjoint pairs in $\cA\times \cB$ via their zero entries. More precisely, recall that the intersection matrix of a pair of sets systems $(\cA,\cB)$ is the matrix whose rows and columns are indexed by the elements of $\cA$ and $\cB$ and whose entries are defined as $M(A,B)=|A\cap B|$. 
In the previous sections, we presented a number of results which lead us to completely understand the size of all-zero submatrices in those low rank matrices that arise from set systems. 

In this section, our goal is to present results that address more general matrices. These results are motivated by both the log-rank conjecture and Conjecture \ref{conj:lovett}. Our first result generalizes the best known upper bound on the log-rank conjecture and shows that Conjecture \ref{conj:lovett} holds for a large class of matrices.

\begin{theorem}\label{thm:sparse_matrix}
Let $M$ be an $m\times n$ matrix of rank $r$ with non-negative entries, none of which are in $(0, 1)$. If the average of the entries of $M$ is $\eps\leq 1/2$, then $M$ contains an all-zero submatrix of size at least $$n2^{-O(\log r+\sqrt{\eps r})}.$$
\end{theorem}

If $M$ is a binary matrix, then either at least half of the entries are zero, or at least half of the entries are 1. Hence, we can apply the previous theorem to either $M$ or $J-M$ to conclude that $M$ contains a constant submatrix of size at least $n2^{-O(\sqrt{r})}$, recovering the best known bound of Sudakov and Tomon \cite{ST24} on the log-rank conjecture. On the other hand, an important new feature of Theorem \ref{thm:sparse_matrix} is that it is optimal when $\eps\gg (\log r)^2/r$. Indeed, the rank $r$ matrix $M$ in Construction 2 also satisfies that the average of its entries is $\eps$ and it contains no all-zero submatrix of size $ n \cdot \exp(-\sqrt{\eps r})$. 

\medskip

Using similar methods, we also establish the following result, which generalizes the best known bound for log-rank conjecture from binary matrices to matrices with integer  entries between $0$ and $t$. Furthermore, it extends the previous theorem to matrices with average entry larger than~$1$.

\begin{theorem}\label{thm:few_distinct_values}
Let $M$ be an $m\times n$ matrix of rank $r$ with nonnegative integer entries such that the average of the entries is at most $t\geq 1$. Then $M$ contains a constant submatrix of size at least $2^{-O(t\sqrt{r})}m \times 2^{-O(t\sqrt{r})} n.$
\end{theorem}

\noindent
Observe that for constant $t$, this bound matches the best known bound for the log-rank conjecture. 

% Since the proof of Theorem~\ref{thm:few_distinct_values} uses very similar methods as Theorem~\ref{thm:sparse_matrix} but is more technical, we do not spell it out in this paper. We refer the interested reader to the Appendix of the arXiv version of this paper for the full proof.
\medskip

\noindent
\textbf{Organization.}  We start with the proof of Theorem \ref{thm:singer sudan sparse} and Lemma \ref{lemma:covering lemma} in Section \ref{sect:singer-sudan}, since their proofs are the shortest and, we believe, the most elegant. Then, we build on the ideas of Section~\ref{sect:singer-sudan} and present the proof of Theorems \ref{thm:coarse erdos daykin}, \ref{thm:fine-grained daykin erdos}, \ref{thm:daykin erdos nonzero intersection} and \ref{thm:even} about the Daykin-Erd\H{o}s problem in Section~\ref{sect:daykin-erdos}. The proof of Theorem~\ref{thm:sparse_matrix} is presented in Section \ref{sect:low-rank matrices}. Finally, in Section \ref{sect:construction}, we prove that the constructions discussed in the introduction indeed have the claimed properties. We conclude by discussing some open problems in Section~\ref{sect:open_problems}. 
Finally, we also provide the proof of Theorem~\ref{thm:few_distinct_values} in the Appendix, since its proof uses similar methods as Theorem~\ref{thm:sparse_matrix}, but is more technical.

\section{Large bicliques in disjointness graphs}\label{sect:singer-sudan}

In this section, we prove  Theorem~\ref{thm:singer sudan sparse} and our covering result, Lemma~\ref{lemma:covering lemma}. We use the following standard graph theoretic notation. If $G$ is a graph and $x\in V(G)$, then $N_G(x)=N(x)$ denotes the set of neighbours of $x$. Also, if $U\subseteq V(G)$, then $N_G(U)=N(U)$ is the \emph{common neighbourhood} of $U$, that is, the set of vertices $y\in V(G)\setminus U$ that are connected to every element of $U$ by an edge. In case $U=\{u_1,\dots,u_s\}$, we simply write $N(u_1,\dots,u_s)$ instead of $N(U)$.

Our approach is inspired by the argument of Alon and Frankl which proved the original Daykin-Erd\H{o}s conjecture. Their argument proceeds as follows: given a family $\cF\subset 2^{[n]}$ with a dense disjointness graph, take $k$ random sets $A_1, \dots, A_k$ from the family $\cF$ with a suitably chosen $k$ and consider their common neighbourhood. Since the disjointness graph of $\cF$ is dense, this common neighbourhood is large, i.e. many sets of $\cF$ are disjoint from $A_1, \dots, A_k$, and so the union $U_A=A_1\cup\dots\cup A_k$ cannot be much larger than $n/2$. If the family $\cF$ is large, a simple counting argument shows that the union $A_1\cup\dots\cup A_k$ is expected to be significantly larger than $n/2$, giving a contradiction.
In our setting, the family $\cF$ is not known to be large and so the union $A_1\cup\dots\cup A_k$ may not be large, but we can show that it still covers many sets of $\cF$, which suffices to prove Theorem~\ref{thm:singer sudan sparse}. To show this, we introduce the notion of \textit{bad} sets, which is the key new ingredient in the proof.

\begin{proof}[Proof of Theorem~\ref{thm:singer sudan sparse}.]
Define a bipartite graph $G$ on the vertex set $\cA\cup \cB$, in which $A\in \cA$ is adjacent to $B\in \cB$ if $A\cap B=\varnothing$. Then $G$ is the disjointness graph of $(\cA,\cB)$, and by the assumptions of the theorem, $G$ has at least $\delta|\cA||\cB|$ edges. Throughout the proof, we assume that $|\cA|, |\cB|\geq 2^{2\sqrt{n\log{\delta^{-1}}}}$. Otherwise, if $|\cB|\leq 2^{2\sqrt{n\log{\delta^{-1}}}}$, then pick $\cB'=\{B\}$ and $\cA'=N(B)$, where $B\in \cB$ is the vertex of maximum degree in $G$. Then  $|\cB'|=1\geq 2^{-2\sqrt{n\log{\delta^{-1}}}}|\cB|$ and $|\cA'|=\delta |\cA|\geq 2^{-\sqrt{n\log{\delta^{-1}}}}|\cA|$. This in particular implies $\delta \ge 2^{-n}$.

The main idea of the proof is to pick $k=\lfloor \sqrt{n/\log \delta^{-1}}\rfloor$ sets $A_1, \dots, A_k$ uniformly at random from $\cA$ and consider their union $U=\bigcup_{i=1}^k A_i$. Then set $\cR$ to be the family of sets in $\cA$ contained in $U$ and set $\cS$ be the family of sets in $\cB$ contained in $[n]\backslash U$. This construction ensures that for each pair of sets $(R,S)\in \cR\times \cS$, $R$ and $S$ are disjoint, and therefore our main goal is to show that with positive probability, both $\cR$ and $\cS$ are sufficiently large. Note that $|\cR|$ is precisely the number of sets of $\cA$ covered by $U$, while $|\cS|$ is equal to the size of the common neighbourhood of vertices $A_1, \dots, A_k$ in the disjointness graph $G$, i.e. $|\cS|=|N_{G}(A_1, \dots, A_k)|$.

First, we show that $$\Pb[|\cR|\geq 2^{-2n/k} |\cA|]\geq 1-2^{-n}.$$ Say that a set $U$ is \textit{bad} if it covers less than $2^{-2n/k}|\cA|$ sets from the family $\cA$. Then the probability that $k$ randomly chosen sets from $\cA$ land inside $U$ is at most $(2^{-2n/k})^k=2^{-2n}$. Hence, taking the union bound over all $2^n$ potential bad sets, the probability that $U=\bigcup_{i=1}^k A_i$ is bad can be bounded by $2^{-n}$. Therefore, with probability at least $1-2^{-n}$,  $|\cR|\geq 2^{-2n/k} |\cA|$.

Next, we show that $$\Pb[|\cS|\geq \delta^k |\cB|/2]> 2^{-n}.$$
Observe that the expected size of the common neighbourhood of $k$ random vertices $A_1, \dots, A_k\in \cA$ in the graph $G$ can be computed as follows
\[\E\big[|N_G(A_1, \dots, A_k)|\big]=\sum_{B\in \cB} \Pb\big[B\in N_G(A_1, \dots, A_k)\big]=\sum_{B\in \cB} \left(\frac{\deg_G(B)}{|\cA|}\right)^k\geq \delta^k |\cB|,\]
where the last inequality follow from Jensen's inequality. Using that $0\leq |\cS|\leq 2^n$, we can write \begin{equation}\label{eq: common neighborhood}\delta^k |\cB|\leq \E\big[|\cS|\big]\leq  2^n\cdot \Pb\Big[|\cS|\geq \delta^k |\cB|/2\Big] + \frac{\delta^k |\cB|}{2}\cdot \Pb\Big[|\cS|< \delta^k |\cB|/2\Big] < 2^n \Pb\Big[|\cS|\geq \delta^k |\cB|/2\Big]+\frac{\delta^k |\cB|}{2}.\end{equation}
Note that $\delta^k|\cB|\geq \delta^{\sqrt{n/\log\delta^{-1}}}2^{2\sqrt{n\log\delta^{-1}}} = 2^{\sqrt{n\log\delta^{-1}}}\geq 2$, where the last inequality holds by our assumption $\delta\leq 1-1/n$. Therefore, comparing the two sides of Eq.~\ref{eq: common neighborhood} gives $\Pb[|\cS|\geq \delta^k |\cB|/2]>2^{-n}$. 

Hence, there exists a choice of $A_1,\dots,A_k$ for which 
\begin{align*}
|\cS|&\geq \delta^k |\cB|/2=2^{-k\log\delta^{-1}-1}|\cB|\geq 2^{-\sqrt{n\log\delta^{-1}}-1}|\cB|,\\
|\cR|&\geq 2^{-2n/k}|\cA|\geq 2^{-4\sqrt{n\log\delta^{-1}}}|\cA|,
\end{align*}
and this completes the proof.
\end{proof}

We now use the notion of bad sets from the previous proof to show Lemma~\ref{lemma:covering lemma}.

\begin{proof}[Proof of the Lemma~\ref{lemma:covering lemma}.]
Say that a set $U\subseteq [n]$ is \textit{bad} if $\mu(2^{U})=\Pb[A\subseteq U]\leq 2^{-n/r-1}$, where $A$ is a random set drawn from $\mu$. If  $A_1, \dots, A_r$ are drawn independently from the distribution $\mu$, then for every bad set $U$,
$$\Pb[A_1, \dots, A_r\subseteq U]=\prod_{i=1}^r \Pb(A_i\in 2^{U})\leq (2^{-n/r-1})^r\leq 2^{-n-1}.$$
Hence, by the union bound over all subsets of $[n]$, the probability that $B=A_1\cup \dots\cup A_r$ is a bad set is at most $2^n \cdot 2^{-n-1}=1/2$. Furthermore, if $B$ is not bad, a random set $A_0$, chosen from the distribution $\mu$, lies in $2^B$ with probability at least $2^{-n/r-1}$. In conclusion, $\Pb[A_0\subseteq A_1\cup\dots\cup A_r]\geq 1/2\cdot 2^{-n/r-1}=2^{-n/r-2}$, completing the proof.
\end{proof}

\section{Optimal bounds for the Daykin - Erd\H os problem}\label{sect:daykin-erdos}

In this section, we prove our result concerning the Daykin-Erd\H{o}s problem. We present two proofs of Theorem \ref{thm:coarse erdos daykin}, as both proofs are fairly short and the techniques might be of independent interest. Our first proof is combinatorial and, in our opinion, more intuitive because of the similarities with the proof of Theorem~\ref{thm:singer sudan sparse}. The second approach, based on entropy methods, is a bit stronger, since it can also be used to show tight bounds in Theorem~\ref{thm:fine-grained daykin erdos}.

\subsection{A combinatorial proof}

Our proof uses ideas from Section~\ref{sect:singer-sudan}, together with the following new ingredient. Instead of one, we find two collections of sets, $C_A=\{A_1, \dots, A_t\}\subseteq \cF$ and $C_B=\{B_1, \dots, B_s\}\subseteq \cF$, such that both collections have many common neighbours in the disjointness graph, the unions $U_A=A_1\cup\dots\cup A_t$ and $U_B=B_1\cup\dots\cup B_s$ are disjoint and, moreover, these unions cannot be significantly enlarged by adding a single set to $C_A$ or $C_B$.
We then complete the proof by giving an upper bound on the size of the common neighbourhood of $C_B$,
by observing that all of the common neighbours of $C_B$ are disjoint from $U_B$ and do not have many elements outside $U_A$. 

We prepare the proof with a folklore result about finding subgraphs of large minimum degree. We prove this for the sake of completeness.

\begin{lemma}\label{lemma:cleaning}
Let $H$ be a bipartite graph with parts $X$ and $Y$ and with at least $\eps |X||Y|$ edges. Then there exist sets $X'\subseteq X, Y'\subseteq Y$ such that $|X'||Y'|\geq \frac{\eps}{2}|X||Y|$ and we have the minimum degree conditions $|N(x)\cap Y'|\geq \eps |Y|/4, |X'\cap N(y)|\geq \eps |X|/4$ for all $x\in X', y\in Y'$.
\end{lemma}
\begin{proof}
We perform the following algorithm on the graph $H$ in order to eliminate the low-degree vertices. If a vertex $x\in X$ has degree smaller than $\eps |Y|/4$, we remove it from the graph and similarly, if a vertex $y\in Y$ has degree smaller than $\eps |X|/4$, we also remove it from the graph. We repeat this procedure until no vertex has degree smaller than $\eps |Y|/4$ (for the vertices of $X$), or $\eps |X|/4$ (for the vertices of $Y$). We denote by $X', Y'$ the sets of remaining vertices. By definition, these satisfy the minimum degree condition, so it remains to check that $|X'||Y'|$ is still large. The number of edges erased by removing vertices of $X$ can be bounded by $\frac{\eps}{4} |X||Y|$, and similarly for the vertices of $Y$. Hence, at least $\eps|X||Y|-2\cdot \frac{\eps}{4}|X||Y|\geq \eps|X||Y|/2$ edges remain in the graph. Therefore, we must have $|X'||Y'|\geq \frac{\eps}{2}|X||Y|$, completing the proof.
\end{proof}

We now discuss the proof of Theorem~\ref{thm:coarse erdos daykin}, which states that the maximum number of disjoint pairs in a set family $\cF$ of size $m=2^{(1/2+\delta)n}$ is $m^{2-c\delta/\log_2\delta^{-1}}$. Since the upper bound given by Theorem~\ref{thm:coarse erdos daykin} in the regime $\delta = O(\frac{\log n}{n})$ is worse than the trivial upper bound of $\binom{m}{2}$, we focus on the regime where $\delta\geq 10^4\log n/n$. Also, we restate and prove Theorem~\ref{thm:coarse erdos daykin} in the bipartite setting, since our proof works more naturally in this setting. This is in no way leads to loss of generality, since Theorem~\ref{thm:coarse erdos daykin} can be derived from Theorem~\ref{thm:coarse bipartite erdos daykin} by randomly splitting a family $\cF$ into two equally-sized parts $\cA, \cB$ without decreasing the density of disjoint pairs.

\begin{theorem}\label{thm:coarse bipartite erdos daykin}
Let $\cA, \cB$ be set families on the ground set $[n]$ of size $m=2^{(1/2+\delta)n}$, where $\delta$ is a parameter satisfying $10^4\frac{\log n}{n}\leq \delta\leq 1/2$. Then the density of their disjointness graph is at most $$d(\cA, \cB)\leq 2^{-\frac{\delta}{100\log \delta^{-1}}n}.$$
\end{theorem}
\begin{proof}
Let $G$ be the disjointness graph of families $\cA$ and $\cB$ and assume that its density $\eps$ satisfies $\eps> 2^{-\frac{\delta}{100\log 1/\delta}n}$. Our first step is to apply Lemma~\ref{lemma:cleaning} to clean the graph $G$ and obtain subfamilies $\cA_1\subseteq\cA$ and $\cB_1\subseteq \cB$ of size $|\cA_1||\cB_1|\geq \frac{\eps}{2} |\cA||\cB|$ which satisfy $|N_G(A)\cap \cB_1|\geq \eps|\cB|/4, |\cA_1\cap N_{G}(B)|\geq \eps|\cA|/4$ for all $A\in \cA_1, B\in \cB_1$. In particular, we have $|\cA_1|\geq \frac{\eps}{2}|\cA|$ and $|\cB_1|\geq \frac{\eps}{2}|\cB|$.
Finally, let $G_1$ be the induced subgraph of $G$ on the vertex set $\cA_1\cup \cB_1$.

Let $\rho=0.9$, and let $t$ be the maximal integer for which there exists a $t$-tuple $(A_1, \dots, A_t)$ of sets in $\cA_1$ with the following two properties:
\begin{itemize}
    \item $|N_{G_1}(A_1, \dots, A_i)|\geq \left(\frac{\eps}{64}\right)^i|\cB_1|$ for all $i=1, \dots, t$, and
    \item $\Big|A_1\cup\dots\cup A_i\Big|\geq \left(\frac{1}{2}-\rho^i\right) n$, for all $i=1, \dots, t$.
\end{itemize} 

First of all, note that $t\geq 1$ as every 1-tuple $(A_1)$ satisfies the properties trivially. We also claim that $t\leq 16\log_2 \delta^{-1}$. Assume to the contrary that $t>16\log_2 \delta^{-1}$ and let $N$ be the number of sets in $\cB_1$ disjoint from $A_1\cup\dots\cup A_{k}$, where $k=16\log_2\delta^{-1}$. There are at least $(\eps/64)^k |\cB_1|\geq (\eps/64)^{k+1}2^{(1/2+\delta)n}$ such sets. On the other hand, all these sets are contained in $[n]\backslash (A_1\cup\dots\cup A_k)$, which is a set of size at most $(1/2+\rho^k)n$ by the second property, implying 
\[(\eps/64)^{k+1}2^{(1/2+\delta)n}\leq N\leq  2^{(1/2+\rho^k)n}.\]
However, the following short computation shows that our parameters have been chosen in such a way that this inequality is impossible. Observe that $\rho^k=0.9^{8\cdot 2\log_2\delta^{-1}}\leq (1/2)^{2\log_2\delta^{-1}}\leq \delta^2\leq \delta/2$, where we used $0.9^8\approx 0.43<1/2$. Therefore, the above inequality implies $(\eps/64)^{k+1}\leq 2^{-\delta n/2}$. Replacing $k+1$ by $2k$ and taking logarithms, we find that $2k\log_2(64/\eps)\geq \delta n/2$, where $\log_2(1/\eps)\leq \frac{\delta n}{100\log_2\delta^{-1}}$ by assumption. Plugging in $k=16\log_{2}\delta^{-1}$, we find that $\frac{32}{100}\delta n+32\log_{2}\delta^{-1}\cdot \log_2 64\geq \delta n/2$. But this is impossible, since the second term is bounded by $6\cdot 32\log_{2}\delta^{-1}\leq \frac{192}{10^4}\delta n$ (where we have used $\delta\geq 10^4\frac{\log n}{n}$) and $\frac{32}{100}+\frac{192}{10^4}<1/2$. Hence, we conclude $t\leq 16\log_2 \delta^{-1}$.

Each vertex in $N_{G_1}(A_1, \dots, A_t)$ sends at least $\eps |\cA|/4\geq \eps |\cA_1|/4$ edges to $\cA_1$. Hence, the density of the subgraph induced between $\cA_1$ and $N_{G_1}(A_1, \dots, A_t)$ is at least $\eps/4$. By applying Lemma~\ref{lemma:cleaning} to this subgraph we can find sets $\cA_2\subseteq \cA_1$ and $\cB_2\subseteq N_{G_1}(A_1, \dots, A_t)$ such that \[|\cA_2||\cB_2|\geq \frac{\eps}{8}|\cA_1||N_{G_1}(A_1, \dots, A_t)|\geq \Big(\frac{\eps}{64}\Big)^{t+1} |\cA_1||\cB_1|\geq \Big(\frac{\eps}{64}\Big)^{t+2} |\cA||\cB|\]  and the subgraph $G_2$ induced by $\cA_2\cup \cB_2$ satisfies $|N_{G_2}(A)\cap \cB_2|\geq \eps |N_{G_1}(A_1,\dots,A_t)|/16, |\cA_2\cap N_{G_2}(B)|\geq \eps |\cA_1|/16$ for all $A\in \cA_2, B\in \cB_2$. Furthermore, we have $|\cA_2|\geq (\eps/8) |\cA_1|\geq (\eps/8)^2 |\cA|$.

Observe that for every $A\in \cA_2$, we have $|A\backslash (A_1\cup\dots\cup A_t)|\leq (\rho^t-\rho^{t+1})n=\rho^t n/10$, since otherwise $A$ could have been added to the $t$-tuple $(A_1, \dots, A_t)$ and make a larger one, contradicting the maximality of $t$. Moreover, $\Big|\bigcup_{i=1}^t A_i\Big|\leq (1/2-\rho^{t+1})n$, since otherwise any set $A\in \cA_2$ could be added to the $t$-tuple to extend it.

Now, we consider the other side. Let $s$ be maximal for which there exists an $s$-tuple $(B_1, \dots, B_s)$ of elements of $\cB_2$ satisfying the following two properties:
\begin{itemize}
    \item $|N_{G_2}(B_1, \dots, B_i)|\geq \left(\frac{\eps}{64}\right)^i|\cA_2|$ for all $i=1, \dots, s$, and
    \item $\Big|B_1\cup\dots\cup B_i\Big|\geq \left(\frac{1}{2}-\rho^i\right) n$, for all $i=1, \dots, s$.
\end{itemize}
We now wish to use a similar argument as before to conclude that $s\leq 16\log_2\delta^{-1}$. The only difference from the previous calculation is that, if $k=16\log_2\delta^{-1}$, we now have $(\eps/64)^k|\cA_2|\geq (\eps/64)^{k+2}|\cA|$ sets in $\cA_2$ disjoint from $B_1, \dots, B_k$. However, since we end up replacing this $k+2$ with $2k$ anyways, this has no effect on the chain of inequalities presented above.

Furthermore, every vertex of $N_{G_2}(B_1, \dots, B_s)$ sends at least $\eps|\cB_2|/16$ edges to $\cB_2$. Applying Lemma~\ref{lemma:cleaning} to the subgraph induced by $N_{G_2}(B_1, \dots, B_s)$ and $\cB_2$, which has density at least $\eps/16$, we  find subsets $\cA_3\subseteq N_{G_2}(B_1, \dots, B_s)$ and $\cB_3\subseteq \cB_2$ such that the subgraph $G_3$ induced on $\cA_3\cup \cB_3$ satisfies
\[|\cA_3||\cB_3|\geq \frac{\eps}{32}|N_{G_3}(B_1, \dots, B_s)||\cB_2|\geq \Big(\frac{\eps}{64}\Big)^{s+1}|\cA_2||\cB_2|\geq \Big(\frac{\eps}{64}\Big)^{s+t+3}|\cA||\cB|\] and $|N_{G_3}(A)\cap \cB_3|\geq \eps |\cB_3|/64, |\cA_3\cap N_{G_3}(B)|\geq \eps |N_{G_2}(B_1, \dots, B_s)|/64$ for all $A\in \cA_3, B\in \cB_3$.

Finally, observe again that each set $B\in \cB_3$ satisfies $|B\backslash (B_1\cup\dots\cup B_s)|\leq \rho^s n/10$ and $\Big|\bigcup_{i=1}^s B_i\Big|\leq (1/2-\rho^{s+1})n$, for otherwise $B$ could have been added to extend the $s$-tuple $(B_1, \dots, B_s)$.

Now comes the endgame of our proof. We assume that $t\leq s$, the other case can be handled in an identical manner. Write $U_A=\bigcup_{i=1}^t A_i, U_B=\bigcup_{i=1}^s B_i$ and recall that $|U_A|, |U_B|\geq (1/2-\rho^t) n$, $|U_A|\leq (1/2-\rho^{t+1})n$ and $U_A\cap U_B=\varnothing$. Every set $A\in \cA_3$ is contained in $[n]\backslash U_B$ and satisfies $|A\backslash U_A|\leq \rho^tn/10$. Hence, the size of $\cA_3$ can be upper bounded by the total number of sets satisfying these two conditions, which is
\begin{align*}
|\cA_3| \leq 2^{|U_A|}\sum_{i=0}^{\rho^tn/10}\binom{n-|U_A|-|U_B|}{i}\leq 2^{(1/2-\rho^{t+1})n}\sum_{i=0}^{\rho^tn/10}\binom{2\rho^t n}{i}\leq 2^{(1/2-\rho^{t+1})n+2\rho^t n H(\frac{1}{20})}.
\end{align*}
In the last inequality  we  used the estimate $\sum_{i=1}^a\binom{b}{i}\leq 2^{H(a/b)b}$, where  $H(x)=-x\log_2 x-(1-x)\log_2(1-x)$ is the binary entropy function. This estimate holds for every $a\leq b/2$, which is satisfied in our case, where $a=\rho^tn/10$ and $b=2\rho^t n$. But observe that $1/2-\rho^{t+1}+2\rho^t H(1/20)=1/2+\rho^t (-\rho+2H(1/20))<1/2$ (using that $H(1/20)\approx 0.29<\rho/2$), which means that $|\cA_3|\leq 2^{n/2}$. Hence, we get \[2^{n/2}\geq |\cA_3|\geq \left(\frac{\eps}{64}\right)^{s+t+3}|\cA|\geq \left(\frac{\eps}{64}\right)^{35\log \delta^{-1}} 2^{(1/2+\delta)n}.\]
We claim that this is again impossible since our parameters have been chosen such that $(\eps/64)^{35\log_2 \delta^{-1}}2^{\delta n}>1$. To see this, we take logarithms and use $\eps\geq 2^{-\frac{\delta}{100\log_2 1/\delta}n}$ to get
$35\log_2 \delta^{-1} \log_2 \eps + 35 \log_2 \delta^{-1} \log_2 (1/64)+\delta n\geq -\frac{35}{100}\delta n-\frac{35\cdot 6}{10^4} \delta n+\delta n>0$. This shows that $|\cA_3|\leq 2^{n/2}$ is impossible, thus deriving the final contradiction to the assumption $\eps\geq 2^{-\frac{\delta}{100\log_2 1/\delta}n}$ and completing the proof of the theorem.
\end{proof}

\subsection{Entropy methods}

In this section, we present the second approach to Theorem \ref{thm:coarse erdos daykin}, which uses entropy methods. 
The proof is similar in spirit to the previous one, the main novelty being Lemma~\ref{lemma:entropy expected union} which uses entropy to lower-bound the expected size of the union $U_A=A_1\cup\dots\cup A_k$ of $k$ random sets from a large family $\cF$. This lemma is then combined with the dependent random choice method, which allows us to lower-bound the size of the common neighbourhood of $k$ random sets in the disjointness graph. The final contradiction then comes by comparing the size of the union $U_A$ and the number of sets from $\cF$ contained in its complement, i.e. the size of the common neighbourhood of $A_1, \dots, A_k$.

This approach can also be extended to give a proof of Theorem \ref{thm:fine-grained daykin erdos}. In particular, we obtain the following theorem, which implies both.

\begin{theorem}\label{thm:entropy daykin erdos}
Let $\cA, \cB\subset 2^{[n]}$ such that $d(\cA, \cB)\geq 2^{-\theta n}$, where $\theta\in (0, 1)$. Then,  $$|\cA||\cB|\leq \exp(O(\max\{\theta n, \sqrt{\theta n\log_2 n}\}))2^{n(1+\theta\log_2\theta^{-1})}.$$
\end{theorem}

\noindent 
Note that Theorem~\ref{thm:coarse erdos daykin} can be deduced from this theorem by setting $\theta=\delta/\log\delta^{-1}$. We prepare the proof with a technical inequality.

\begin{lemma}\label{lemma:entropy inequality}
There exists an absolute constant $C>0$ such that for every positive integer $k$ and  $p\in [0, 1]$, we have
\[H(p)=p\log_2\frac{1}{p}+(1-p)\log_2\frac{1}{1-p}\leq 1-(1-p)^k+\frac{C}{2^k}.\] 
\end{lemma}
\begin{proof}
We show that $C=e^{40}$ suffices. Since we always have $H(p)\leq 1$, the statement is only interesting when $2^k\geq C$, so we may assume that $k> 40$. For the same reason, we may also assume that $p\leq 1/2$. We split the proof into two cases, based on whether $p\leq 1/k$ or $p\geq 1/k$.

\medskip
\noindent
\textbf{Case 1: $p\leq 1/k$.} On the left-hand side, we bound the second term as $(1-p)\log_2\frac{1}{1-p}\leq \log_2\frac{1}{1-p}\leq \log_2(1+2p)\leq 2p$. On the other hand, using the Bonferroni inequalities (or equivalently the inclusion-exclusion inequalities), we know that $(1-p)^k\leq 1-kp+\binom{k}{2}p^2$ and so $1-(1-p)^k\geq 1-(1-kp+\binom{k}{2}p^2)\geq kp-\binom{k}{2}p^2$. These two bounds reduce our inequality to showing that 
\[p\log_2\frac{1}{p}\leq (k-2)p-\binom{k}{2}p^2+\frac{C}{2^k}.\]
If $p\geq 1/k^2$, or equivalently $\log_2 1/p\leq 2\log_2 k$, this inequality is easy to show since $p\cdot 2\log_2 k\leq p\cdot \frac{k-3}{2}\leq p(k-2)-\frac{k(k-1)}{2}p^2$. In the first step, we used that $k>40$ which implies $2\log_2 k\leq (k-3)/2$, while in the second step, we  used that $p\leq 1/k$.
If $p\leq 1/k^2$, we have $\binom{k}{2}p^2\leq p$ and so it suffices to show $p\log_2\frac{1}{p}-p(k-3)\leq C/2^k$, which follows from $\log_2\frac{1}{2^{k-3} p}\leq \frac{C}{2^k p}$ since $C>8$.

\medskip
\noindent
\textbf{Case 2: $p>1/k$.} Here, we have three subcases, all of which are easily solved.

\medskip
\textit{Subcase 2.1: $\frac{1}{k}<p\leq \frac{1}{10}$.} We have $H(p)\leq H\big(\frac{1}{10}\big)<\frac{1}{2}$ and $1-(1-p)^k\geq 1-\big(1-\frac{1}{k}\big)^k\geq 1-\frac{1}{e}>\frac{1}{2}$.

\medskip
\textit{Subcase 2.2: $\frac{1}{10}<p\leq \frac{1}{2}-\frac{5}{k}$.} From $p\geq 1/10$ we get that $1-(1-p)^k\geq 1-0.9^k$. 
Since $H'(1/2)=0$ and $H''(p)=-\frac{1}{p(1-p)\ln 2}\leq -\frac{4}{\ln 2}$, we have that the function $f(p)=
\big(H(1/2)-\frac{2}{\ln 2}\left(1/2-p\right)^2\big)-H(p)$ satisfies $f(1/2)=f'(1/2)=0$ and $f''(p)\geq 0$. Therefore $H(p)$ 
must lie below the parabola $H(1/2)-\frac{2}{\ln 2}\left(1/2-p\right)^2$, i.e. we have
\[H(p)\leq 1-\frac{2}{\ln 2}\left(\frac{5}{k}\right)^2\leq 1-0.9^k\leq 1-(1-p)^k,\]
where we used that $\frac{2}{\ln 2}\left(\frac{5}{k}\right)^2\geq 0.9^k$ for all $k$.

\medskip
\textit{Subcase 2.3: $\frac{1}{2}-\frac{5}{k}\leq p\leq \frac{1}{2}$.} In this case we have $H(p)\leq 1$ and $1-(1-p)^k+\frac{C}{2^k}\geq 1-\frac{(1+10/k)^k}{2^k}+\frac{C}{2^k}\geq 1$, since $C>e^{40}$.
\end{proof}

The main ingredient we need for the proof of Theorem~\ref{thm:entropy daykin erdos} is the following covering lemma, which lower bounds the expected size of the union $U$ of $k$ random sets chosen from a large set family $\cF$ of size $|\cF|=2^{\alpha n}$.
The intuition behind this lemma is that the expected size of this union will be minimized in case the family $\cF=2^{[\alpha n]}$ for some $\alpha\in (0, 1)$. In this case, the expected size of the union $U$ of $k$ random sets in $\cF$ will be $(1-2^{-k})\alpha n$. The following lemma shows that a similar statement is true without assumptions on the structure of the family $\cF$.

\begin{lemma}\label{lemma:entropy expected union}
Let $\cF\subset 2^{[n]}$ of size $2^{\alpha n}$. If $F_1, \dots, F_k$ are random elements of $\cF$, chosen independently from the uniform distribution, then
$$\E\Big[\big|{\textstyle \bigcup\limits_{i=1}^k} F_i\big|\Big]\geq \alpha n-\frac{Cn}{2^k},$$ where $C$ is the absolute constant coming from Lemma~\ref{lemma:entropy inequality}.
\end{lemma}
\begin{proof}
For $i\in [n]$, let  $p_i$ denote the proportion of sets in $\cF$ which contain $i$. By linearity of expectation, $$\E\Big[\big|{\textstyle \bigcup\limits_{i=1}^k} F_i\big|\Big]=\sum_{i=1}^n \Pb\Big[i\in {\textstyle \bigcup\limits_{i=1}^k} F_i\Big]=\sum_{i=1}^n 1-(1-p_i)^k.$$
By Lemma~\ref{lemma:entropy inequality}, we have $1-(1-p_i)^k\geq H(p_i)-C/2^k$ and therefore $$\E\Big[\big|{\textstyle \bigcup\limits_{i=1}^k} F_i\big|\Big]\geq \sum_{i=1}^n H(p_i)-Cn/2^k.$$ Finally, subadditivity of entropy shows that $\sum_{i=1}^n H(p_i)\geq \log |\cF|=\alpha n$ (see e.g. Corollary 15.7.3 of Alon-Spencer \cite{AlonSpencer}). Combining everything, $\E\big[|\bigcup_{i=1}^k F_i|\big]\geq \alpha n-Cn/2^k$.
\end{proof}

We now give the proof of Theorem~\ref{thm:entropy daykin erdos}. 

\begin{proof}[Proof of Theorem \ref{thm:entropy daykin erdos}]
Let $\eps=d(\cA,\cB)\geq 2^{-\theta n}$, let $C$ be the constant given by Lemma \ref{lemma:entropy inequality}, and let $c=\log_2 C$. Assume that $$|\cA||\cB|\geq \exp((c+8)\max\{\theta n, \sqrt{\theta n\log_2 n}\})2^{n+\theta\log_2\theta^{-1}n}.$$ To simplify notation, we write $M=\max\{\theta n, \sqrt{\theta n\log_2 n}\}$. Furthermore, we denote the size of $\cA$ by $|\cA|=2^{\alpha n}$, where $\alpha\in [0, 1]$.

The proof is based on the dependent random choice method. We refer the reader interested in learning more about this method to the survey of Fox and Sudakov \cite{FS09}. Let $t=\max\{\sqrt{\frac{\log_2 n}{\theta n}}, 1\}=\frac{M}{\theta n}$ and $k=\log_21/\theta+c$. Choose sets $B_1, \dots, B_t\in \cB$ uniformly at random and let $\cA_0=N(B_1, \dots, B_t)$ be the common neighbourhood of these sets in the disjointness graph. Then $\E[|\cA_0|]\geq \eps^t |\cA|$ and hence $\E[|\cA_0|^k]\geq \eps^{tk} |\cA|^k$ by convexity. 

Let $s=n-\alpha n+5M$, and say that an ordered $k$-tuple $(A_1, \dots, A_k)\in \cA_0^k$ is \emph{unfriendly} if $|N(A_1,\dots,A_k)|< 2^s$, otherwise say that it is friendly. Let $X$ be the number of unfriendly $k$-tuples. Then 
\begin{align*}
\mathbb{E}[|X|]&\leq |\cA|^k \left(\frac{2^{s}}{|\cB|}\right)^t= |\cA|^k \left(\frac{2^{n+5M}}{|\cA||\cB|}\right)^t\leq |\cA|^k2^{-((c+3)M+\theta\log_2\theta^{-1}n)t}\\
&\leq  |\cA|^k 2^{-(  c\theta n + 3\sqrt{\theta n\log_2 n}+n\theta \log_2 1/\theta)t}\leq |\cA|^k 2^{-n\theta kt} 2^{-3\sqrt{\theta n\log_2 n} \sqrt{\frac{\log_2 n }{ \theta n}}}\leq |\cA|^k \eps^{tk} n^{-3}.
\end{align*}

Since 
$\E\big[|\cA_0|^k - n X \big]\geq \eps^{tk} |\cA|^k-n\eps^{tk}|\cA|^k/n^3\geq \frac{1}{2}\eps^{tk}|\cA|^k$, we conclude that there is an outcome for which $|\cA_0|^k-nX\geq \frac{1}{2}\eps^{tk}|\cA|^k$. Let us fix such an $\cA_0$, then $|\cA_0|\geq \eps^t |\cA|/2$ and more than $(1-1/n)|\cA_0|^k$ of the $k$-tuples in $\cA_0^k$ are friendly.

Say that a $k$-tuple $(A_1, \dots, A_k)\in \cA_0^k$ is \emph{wide}, if $|\bigcup_{i=1}^n A_i|\geq n-s+1$.  Let $(A_1, \dots, A_k)$ be a randomly chosen $k$-tuple from $\cA_0^k$ and let $U=\bigcup_{i=1}^k A_i$ be their union. By Lemma \ref{lemma:entropy expected union}, the expected size of $U$ is at least 
$$\E[|U|]\geq \log_2|\cA_0|-\frac{Cn}{2^k}\geq \log_2|\cA|+t\log_2\eps -1 - \frac{Cn}{2^k}\geq \alpha n-\theta n t -1 - \frac{Cn}{2^k}.$$ By our choice of $k$ and $c$, we have $Cn/2^k=\theta n$ and thus 
$$\E\big[|U|\big]\geq \alpha n-(t+1)\theta n-1\geq \alpha n-2t \theta n-1\geq \alpha n-2M-1\geq n-s+2.$$ Hence, with probability at least $1/n$, we have $|U|\geq n-s+1$. Otherwise, a contradiction comes from the following chain of inequalities $$\E\big[|U|\big]\leq \Pb\big[|U|\geq n-s+1\big]\cdot n+\Pb\big[|U|<n-s+1\big](n-s+1)< \frac{1}{n}\cdot n+\big(1-\frac{1}{n}\big)(n-s+1)\leq n-s+2.$$ 
In other words, at least $\frac{1}{n}|\cA_0|^k$ of the $k$-tuples in $\cA_0^k$ are wide.

In conclusion, there exists a $k$-tuple $(A_1, \dots, A_k)\in \cA_0^k$ which is both friendly and wide. But then  $|N(A_1,\dots,A_k)|>2^{s}$, and writing $U=\bigcup_{i=1}^k A_i$, we have $|U|\geq n-s+1$ as well. This is impossible, as each element of $N(A_1,\dots,A_k)$ is a subset of $[n]\backslash U$. But $[n]\setminus U$ has at most $2^{n-|U|}\leq 2^{s-1}$ subsets, contradiction. 
\end{proof}

Finally, we show that this theorem indeed implies Theorem \ref{thm:fine-grained daykin erdos}.

\begin{proof}[Proof of Theorem \ref{thm:fine-grained daykin erdos}]
   Assume that $n$ is sufficiently large with respect to $d$, let $\theta=(2d-\eps)/n$ for some $\eps\in (0,1/2)$, and let $\cA,\cB\subset 2^{[n]}$ such that $d(\cA,\cB)\geq 2^{-\theta n}=(1+\Theta(\eps))\cdot2^{-2d}$. Then Theorem~\ref{thm:entropy daykin erdos} implies that 
   $$|\cA||\cB|\leq \exp(O(\sqrt{\theta n\log_2n}))2^{n(1+\theta \log \theta^{-1})}.$$
   Here, $\theta\log_2\theta^{-1}\leq \frac{2d-\eps}{n}\log_2 n$,
   so we can write
   $$\exp(O(\sqrt{\theta n\log_2n}))2^{n+\theta\log_2\theta^{-1}n}\leq \exp(O(\sqrt{d\log_2 n}))2^{n}n^{2d-\eps}.$$
   Choosing $\eps=C_d/\sqrt{\log_2 n}$ for some sufficiently large $C_d$ depending only on $d$ and $c_d$, the right hand side is at most $c_dn^{2d}2^n$, finishing the proof. 
\end{proof}

Let us conclude this section by sketching the how the non-bipartite version of Theorem~\ref{thm:fine-grained daykin erdos} can be derived from Theorem~\ref{thm:fine-grained daykin erdos}.  
Suppose that $\cF$ is a family of size $m\geq c_d n^d 2^{n/2}$ containing $\delta \binom{|\cF|}{2}$ disjoint pairs. Note that we may remove all sets of size smaller than $n/10$ from the family $\cF$, since there are $o(2^{n/2})$ such sets. This yields a smaller family $\cF'\subseteq \cF$, with size $(1-o(1))m$ and $(\delta-o(1))\binom{|\cF'|}{2}$ disjoint pairs. 

The main observation is that the disjointness graph $G$ of the family $\cF'$ has $o(m^3)$ triangles. If this is not the case, there is a set $A\in \cF'$ participating in $\Omega(m^2)$ triangles. In other words, the neighbourhood $N_G(A)$ contains $\Omega(m^2)$ edges of the disjointness graphs. Since $|A|\geq n/10$,  $N_G(A)$ is a set family on the ground set of size at most $|[n]\backslash A|\leq 9n/10$, it contains $|N_G(A)|=\Omega(m)$ sets and $\Omega(m^2)$ disjoint pairs. But this contradicts the theorem of Alon and Frankl \cite{AF85}, which shows that such a family has size at most $2^{(1/2+o(1))\cdot 9n/10}$, which is impossible. Hence, $G$ has $o(m^3)$ triangles.

Now, using Szemer\'edi's regularity lemma, we can find an $\eps$-regular partition of $G$ into $K=O_\eps(1)$ parts. Since $G$ has $o(m^3)$ triangles, this partition does not have three parts, such that all three pairs among them are both regular and dense. Applying Mantel's theorem, we conclude that almost all of the edges of $G$ are contained in $K^2/4$ dense regular pairs. Hence, by Pigeonhole principle one of the pairs has density at least $(2-o(1))\delta$. Applying Theorem~\ref{thm:fine-grained daykin erdos} to these two pairs implies that $(2-o(1))\delta\leq (1+o(1))2^{-2d}$ and so $\delta\leq (1+o(1))2^{-2d-1}$.

\subsection{Daykin - Erd\H os problem with nonzero intersections}

In this section, we prove Theorem \ref{thm:daykin erdos nonzero intersection}. The proof relies on a connection with the Approximate Duality Conjecture,  originally proposed by Ben-Sasson and Ron-Zewi \cite{BSRZ15}. Before stating the conjecture, let us introduce the notion of \textit{bias}. For a prime $p$, a $p$-th root of unity $\omega$ and two sets $A, B\subseteq \bF_p^n$, we define the bias of $A$ and $B$ as \[D_\omega(A, B)=\Big|\E_{a\sim A, b\sim B}\big[\omega^{\langle a, b\rangle}\big]\Big|=\frac{\Big|\sum_{a\in A, b\in B}\omega^{\langle a, b\rangle}\Big|}{|A||B|}.\]

\begin{conjecture}[Approximate Duality Conjecture]\label{conj:ADC}
Let $A, B\subseteq \mathbb{F}_p^n$ such that $D_\omega(A, B)\geq \eps$ for some $\omega\neq 1$. Then there exist  $A'\subseteq A, B'\subseteq B$ such that $\langle a,  b\rangle$ is the same for all $a\in A', b\in B'$ and $$|A'||B'|\geq 2^{ -O( \sqrt{ n  \log(1/\varepsilon)} ) } |A||B|.$$
\end{conjecture}

While this conjecture remains open, the following partial progress suffices for our purposes. The following is proved by Ben-Sasson, Lovett and Ron-Zewi \cite{BSLRZ14} for $p=2$, while Bhowmick, Dvir and Lovett \cite{BDL13} extended it for every prime $p$.  

\begin{theorem}[\cite{BSLRZ14}, \cite{BDL13}]\label{thm:weak ADC}
Let $A, B\subseteq \mathbb{F}_p^n$ be such that $|D_\omega(A, B)|\geq \frac{2}{3 p^{3/2}}$ for some $\omega\neq 1$. Then there exist $A'\subseteq A, B'\subseteq B$ such that $\langle a, b\rangle$ is the same for all $a\in A', b\in B'$ and $$|A'||B'|\geq 2^{ -O_p(  n / \log n ) } |A||B|.$$
\end{theorem}

A nice way to think about the bias of two sets is using the Fourier transform. For any integer $n$, prime $p$ and a function $f:\mathbb{F}_p^n\rightarrow \mathbb{C}$, the \emph{discrete Fourier transform} of $f$ is the function $\hat{f}:\mathbb{F}_p^n\rightarrow \mathbb{C}$ defined as
$$\hat{f}(\xi)=\sum_{x\in \mathbb{F}_p^n} \exp\left(\frac{2i\pi}{p}\langle x,\xi\rangle\right) f(x).$$

Let us define $f:\bF_p\to \bR$ such that for $x\in \bF_p$ we have $f(x)=\Pb[\langle a, b\rangle = x]$, where $a, b$ are sampled uniformly at random from $A, B$. If $\omega=\exp(2i\pi k/p)$ is a $p$-th root of unity, the bias of $A, B$ is corresponds to the magnitude of the Fourier coefficient $\hat{f}(k)$, i.e. $D_\omega(A, B)=|\hat{f}(k)|$.

Finally, let us mention one more ingredient from the world of Fourier analysis which will be used. Given $f,g:\mathbb{F}_p^n\rightarrow \mathbb{C}$, we define the usual scalar product $\langle f,g\rangle=\sum_{x\in \mathbb{F}_p^n} f(x)\overline{g(x)}$, and set $\|f\|_2^2=\langle f,f\rangle$. We make use of Parseval's identity which states that 
\[\|\hat{f}\|_2^2=p^n\|f\|_2^2.\]

We now show how this can be used to prove Theorem~\ref{thm:daykin erdos nonzero intersection}

\begin{proof}[Proof of Theorem~\ref{thm:daykin erdos nonzero intersection}.]
Let $p$ be a prime in the interval $[2/c^2, 4/c^2]$, which exists by Bertrand's postulate. Consider the sets $A, B\subseteq \bF_p^n$ which correspond to characteristic vectors of sets from the families $\cA, \cB$. Then the assumption that at least $c|A||B|$ pairs of sets from $\cA, \cB$ have intersection of size $\lambda$ implies $\Pb[\langle a, b\rangle =\lambda \bmod p]\geq c$, where $a\sim A, b\sim B$ are sampled uniformly at random. Let us recall the function $f:\bF_p\to \bR$ defined as $f(x)=\Pb[\langle a, b\rangle = x]$ for every $x\in \bF_p$.

We would like to show that for some $\omega\neq 1$, we have $|D_\omega(A, B)|\geq \frac{2}{3p^{3/2}}$, which is equivalent to showing that the function $f$ has a large Fourier coefficient $\hat{f}(k)$ for some $k\neq 0$.

By Parseval's identity we have $\sum_{k\in \bF_p}|\hat{f}(k)|^2=p\sum_{x\in \mathbb{F}_p}f(x)^2\geq p c^2$. Further, since $\hat{f}(0)=\sum_{x\in \bZ/p\bZ} f(x)=1$ and $p\geq 2/c^2$, we conclude that 
\[\sum_{k\in \bF_p\backslash\{0\}}|\hat{f}(k)|^2 = p c^2-1\geq 1.\]
Hence, for some $p$-th root of unity $\omega=\exp(2i\pi k/p)\neq 1$, we have $D_\omega(A, B)=|\hat{f}(k)|\geq \frac{1}{\sqrt{p}}\geq \frac{2}{3p^{3/2}}$. Therefore, we can apply Theorem~\ref{thm:weak ADC} to conclude that there exist subsets $A'\subseteq A, B'\subseteq B$ such that $\langle a, b\rangle$ is the same for all $a\in A', b\in B'$ and such that $|A'||B'|\geq 2^{-O_c(n/\log n)} |A||B|$.

The sets $A',B'$ correspond to set families $\cA'\subseteq \cA, \cB'\subseteq \cB$ such that the size of the intersection of any two sets from $\cA', \cB'$ gives the same residue modulo $p$. Any such pair of families must satisfy $|\cA'||\cB'|\leq 2^n$, see e.g. Corollary 3.5 in the work of Sgall \cite{Sgall}. Hence, 
\[|\cA||\cB|\leq 2^{O_c(n/\log n)} |\cA'||\cB'|=2^{n(1+o(1))},\]
completing the proof.
\end{proof}

\subsection{Even and odd intersections}

In this section, we prove Theorem \ref{thm:even}, using the discrete Fourier transform on $\bF_2^n$. 

\begin{proof}[Proof of Theorem \ref{thm:even}]
    Let $f,g:\mathbb{F}_2^n\rightarrow \mathbb{R}$ be the characteristic functions of $\mathcal{A}$ and $\mathcal{B}$, respectively. More precisely, each $A\in 2^{[n]}$ naturally corresponds to a vector $v_A\in \mathbb{F}_2^n$, and we write $f(v_A)=1$ if $A\in \cA$, and $f(v_A)=0$ otherwise. The crucial observation is that 
    \begin{align*}        
    \langle \hat{f},g\rangle&=\sum_{x\in\mathbb{F}_2^n} g(x)\sum_{y\in \mathbb{F}_2^n} (-1)^{\langle x,y\rangle}f(y)=\sum_{x\in\mathbb{F}_2^n} \sum_{y\in \mathbb{F}_2^n} g(x)f(y)(-1)^{\langle x,y\rangle}\\
    =&|\{(A,B)\in \mathcal{A}\times \mathcal{B}: |A\cap B|\mbox{ is even}\}|-|\{(A,B)\in \mathcal{A}\times \mathcal{B}: |A\cap B|\mbox{ is odd}\}|
    \geq 2\delta |\cA||\cB|.
    \end{align*}
    On the other hand, by the Cauchy-Schwarz inequality,
    $$  \langle \hat{f},g\rangle\leq \|\hat{f}\|_2\cdot \|g\|_2=2^{n/2}\|f\|_2\cdot \|g\|_2=2^{n/2} |\cA|^{1/2} |\cB|^{1/2}.$$
    Comparing the lower and upper bounds on $\langle \hat{f},g\rangle$, we get the desired bound $|\cA||\cB|\leq 2^{n}/(4\delta^2)$.
\end{proof}

\section{Low rank matrices}\label{sect:low-rank matrices}

In this section, we prove our result concerning low rank matrices, namely Theorem \ref{thm:sparse_matrix}. Our proofs are based on certain notations of matrix discrepancy. To this end, we start by introducing some standard linear algebra notation. 

Given a vector $v$, we write $\|v\|=\|v\|_2$ for the Euclidean norm of $v$. For a matrix $M\in\mathbb{R}^{m\times n}$, $\|M\|_F$ denotes the Frobenius norm of $M$, that is,
$$\|M\|_F^2=\sum_{i,j} M(i,j)^2=\langle M,M\rangle=\sum_{i=1}^n\sigma_i^2,$$
 where $\langle \cdot,\cdot \rangle$ denotes the usual entry-wise dot product on the space of matrices, and  $\sigma_1\geq \dots\geq \sigma_n$ are the singular values of $M$.
 Furthermore, let 
 $$p(M)=\frac{1}{mn}\sum_{i=1}^{n}\sum_{j=1}^m M(i,j)=\frac{1}{mn}\langle M,J\rangle$$
 denote the average of the entries of $M$, where $J=J_{m,n}$ is the $m\times n$ all-ones matrix. 
 Also, we define the \emph{variance} of an $m\times n$ matrix $M$ as 
$q(M)=\frac{1}{mn}\|M-p(M)J\|_F^2 = \frac{1}{mn}\sum_{i, j}(M(i, j)-p(M))^2.$
 
 The \emph{cut-norm} of $M$ is defined as
 $$\|M\|_C=\max_{A\subset [m], B\subset [n]} |\langle M[A\times B],J\rangle|,$$
 and we define the \emph{discrepancy}  of $M$ as 
 $$\disc(M)=\max_{A\subset [m], B\subset [n]} |A||B|\cdot |p(M[A\times B]) - p(M)|=\|M-p(M)J\|_C.$$
 In other words, $\disc(M)$ measures the maximum deviation of the sum of entries of a submatrix from its expected value. Here, we collect some basic properties of the discrepancy.
\begin{claim}\label{claim:half}
   Let $M$ be an $m\times n$ matrix. Then,
   \begin{itemize}
       \item there exists and $m/2\times n/2$ submatrix $M'$ such that $p(M')\leq p(M)-\frac{\disc(M)}{3mn},$ and
       \item for every $m/2\times n/2$ submatrix $M'$, we have $|p(M')-p(M)|\leq\frac{4\disc(M)}{mn}$.
   \end{itemize}
\end{claim}

Since the proof of the first half of this claim is almost identical to that of Claim 2.2 in \cite{ST24}, and the second part is trivial, we omit its proof for the sake of compactness.

\medskip
Here is the main idea in the proof of Theorem~\ref{thm:sparse_matrix} - we use tools from linear algebra to iteratively find sequence of submatrices of $M=M_0\supset M_1\supset \dots$ with smaller and smaller average entry. Once the average of entries becomes $O(1/r)$, a simple greedy argument can be used to construct a large all-zero matrix. To construct this sequence of submatrices, we show a lower bound on the discrepancy of the matrix $M$.

In \cite{ST24}, the last two authors proved that if $M$ is a binary matrix of small rank, then its discrepancy is large. Here, we extend this result to all matrices. First, we prove the following lower bound on the cut-norm, which is the key technical result underpinning all upcoming proofs. The proof of this lemma is based on an elegant approach of Matthew Kwan and Lisa Sauermann, to whom we are indebted for letting us use their idea.

\begin{lemma}\label{lemma:cut}
    There exists $c>0$ such that the following holds. Let $M$ be an $m\times n$ matrix of rank $r$. Then there exists an $m/2\times n/2$ submatrix $M'$ of $M$ such that 
    $$\|M\|_C\geq \frac{c\sqrt{mn}}{\sqrt{r}}\cdot \frac{\|M'\|_F^2}{\|M\|_F}.$$
\end{lemma}
\begin{proof}
Consider the $\gamma_2^*$-norm of $M$, defined as  
$$\gamma_2^{*}(M)=\max \Big|\sum_{i=1}^{m}\sum_{j=1}^n M(i,j)\langle x_i,y_j\rangle\Big|,$$
where the maximum is taken over all vectors $x_1, \dots, x_m, y_1, \dots, y_n\in \mathbb{R}^{d}$ and all dimensions $d$ such that $\|x_i\|, \|y_j\|\leq 1$. The $\gamma^*_2$-norm is a semidefinite relaxation of the cut-norm, and it is well known that Grothendieck's inequality \cite{Groth} implies   $\|M\|_C=\Omega( \gamma_2^*(M))$. See e.g. \cite{LS} for more details.

The first idea of the proof is to use the Singular Value Decomposition theorem, which states that every $m\times n$ real matrix $M$ of rank $r$ can be represented as $M=U\Sigma V^T$, where $U$ is an $m\times r$ real matrix whose columns are orthogonal of unit norm, $\Sigma$ is an $r\times r$ real diagonal matrix, whose diagonal entries are the singular values of $M$, and $V$ is an $n\times r$ real matrix whose rows  are orthogonal of unit norm. We denote the singular values of $M$, i.e. the diagonal entries of $\Sigma$, by $\sigma_k=\Sigma(k, k)$, and let their sum be $\sigma=\sum_{k=1}^r \sigma_k$. We remark that $\sigma$ is also known as the \emph{nuclear norm} or \emph{trace norm} of $M$.

Define the vectors $u_i, v_j\in \mathbb{R}^r$ for $i=1,\dots,m$ and $j=1,\dots,n$ by setting $u_i(k)=\sqrt{\sigma_k}U(i,k)$ and $v_j(k)=\sqrt{\sigma_k}V(j,k)$. Observe that $\langle u_i,v_j\rangle=\sum_{k=1}^r U(i, k)\Sigma(k, k)V(j, k) = M(i, j)$.

Moreover, we have $\sum_{i=1}^m\|u_i\|^2=\sum_{k=1}^r \sigma_k \sum_{i=1}^m U(i,k)^2=\sum_{k=1}^r \sigma_k=\sigma$, where we used that the columns of $U$ have unit norm. This implies that the number of $i\in [m]$ for which $\|u_i\|^2\geq \frac{2\sigma}{m}$ is at most $m/2$, so there exists a subset $I\subset [m]$ of size $m/2$ such that $\|u_i\|^2\leq \frac{2\sigma}{m}$ for every $i\in I$. Similarly, there exists $J\subset [n]$ of size at  $n/2$ such that $\|v_j\|^2\leq \frac{2\sigma}{n}$ for every $j\in J$.

Let $M'=M[I\times J]$ and define the vectors $x_i=\sqrt{\frac{ m }{2\sigma}} u_i$ for $i\in I$, $x_i=0$ for $i\in [m]\setminus I$, and similarly, $y_j=\sqrt{\frac{ n }{2\sigma}} v_j$ for $j\in J$, and $y_j=0$ for $j\in [n]\setminus J$. Observe that $\|x_i\|^2,\|y_j\|^2\leq 1$. Therefore,
 \begin{align*}
     \gamma_2^*(M)&\geq \sum_{i=1}^m\sum_{j=1}^n M(i,j)\langle x_i,y_j\rangle =\sum_{i\in I}\sum_{j\in J} M(i,j)\frac{\sqrt{mn}}{2\sigma}\langle u_i,v_j\rangle\\
     &=\frac{\sqrt{mn}}{2\sigma}\sum_{i\in I}\sum_{j\in J}M(i,j)^2=\frac{\sqrt{mn}}{2\sigma}\|M'\|_F^2.
 \end{align*}

To complete the proof, we need to upper bound $\sigma$, which we do using the Cauchy-Schwartz inequality. Recall that $\sum_{k=1}^r \sigma_k^2=\|M\|_F^2$, and therefore 
\[\sigma=\sum_{k=1}^r\sigma_k\leq r\bigg(\frac{\sum_{k=1}^r\sigma_k^2}{r}\bigg)^{1/2}=\sqrt{r}\|M\|_F.\]
In conclusion,
$\|M\|_C=\Omega(\gamma_2^*(M))\geq \Omega\Big(\frac{ \sqrt{mn}}{2\sqrt{r}}\cdot \frac{\|M'\|_F^2}{\|M\|_F}\Big).$
\end{proof}

Our next goal is to convert Lemma~\ref{lemma:cut} into a statement about discrepancy of low-rank matrices. However, before we do that, we need the following simple claim.

\begin{claim}\label{claim:variance}
For every real $t$, $\|M-tJ\|_F^2\geq q(M)mn$.
\end{claim}
\begin{proof}
By the definition of the Frobenius norm, we have \[\|M-tJ\|_F^2=\sum_{i=1}^m\sum_{j=1}^n (M(i, j)-t)^2=\sum_{i=1}^m\sum_{j=1}^n M(i, j)^2-2 \sum_{i=1}^m\sum_{j=1}^n M(i, j) t+t^2 mn.\] Thus, the quadratic function $t\mapsto \|M-tJ\|_F^2$ is minimized if $t=2 \sum_{i, j} M(i, j)/2mn=p(M)$, in which case we have $\|M-tJ\|_F^2=\|M-p(M)J\|_F^2= q(M)mn$
\end{proof}

\begin{lemma}\label{lemma:disc}
There exists $c>0$ such that for every $m\times n$ matrix $M$ of rank $r$, there is an $m/2\times n/2$ submatrix $M'\subseteq M$ with 
$$\disc(M)\geq cmn\cdot \frac{q(M')}{\sqrt{rq(M)}}.$$
\end{lemma}
\begin{proof}
Recall that $\disc(M)=\|M-p(M)J\|_C$. Therefore, if we set $M_0=M-p(M)J$, then $\disc(M)=\|M_0\|_C$ and $q(M)mn=\|M_0\|_F^2$. Applying Lemma~\ref{lemma:cut} to the matrix $M_0$, we get that there is a half-sized submatrix $M'\subseteq M$ such that
\[\disc(M)=\|M_0\|_C\geq \frac{c\sqrt{mn}}{\sqrt{r}} \frac{\|M'-p(M)J\|_F^2}{\|M_0\|_F}\geq \frac{c\sqrt{mn}}{\sqrt{r}} \frac{q(M')mn/4}{\sqrt{q(M)mn}}=\frac{cmn}{4}\cdot \frac{q(M')}{\sqrt{rq(M)}},\]
where we have used that $\|M'-p(M)J\|_F^2\geq q(M')mn/4$ by Claim~\ref{claim:variance}.
\end{proof}

Our next goal is to set up the proof of Theorem~\ref{thm:sparse_matrix} through three preparatory lemmas.
Throughout the section, we say that $M\in \mathbb{R}^{m\times n}$ is \emph{separated} if no entry of $M$ is in the open interval $(0, 1)$. First, we show that the variance of such matrices is large. 

\begin{lemma}\label{lemma:large_F}
Let $M$ be a separated $m\times n$ matrix with $0\leq p(M)\leq 0.9$. Then $q(M)\geq p(M)/100$.
\end{lemma}
\begin{proof}
Let $p=p(M)$ be the average entry of $M$ and define the function $f(x) = \frac{(x-p)^2}{x} = x-2p+p^2/x$. Since $f'(x) = 1-p^2/x^2$, the function $f$ is increasing when $x\geq 1$, and so for $x\geq 1$ we have $f(x)\geq f(1)=(1-p)^2\geq 1/100$. Therefore,
\[q(M)mn \ge \sum_{M(i,j)\ge 1} (M(i,j)-p)^2 \ge \sum_{M(i,j)\ge 1}M(i,j)f(M(i, j)) \ge f(1)pmn\geq \frac{p}{100} mn.\]
In the above inequality, we have used that $\sum_{M(i,j)\ge 1}M(i,j)\ge \sum_{i, j}M(i,j) =pmn$, since all entries with $M(i, j)<1$ are actually nonpositive. 
\end{proof}

The core of the iteration is given by the following lemma, which is a consequence of Lemma~\ref{lemma:disc}. In essence, the next lemma states that a low-rank separated matrix either contains a half-sized submatrix in which the average entry decreases significantly, or a half-sized submatrix in which the variance drops tremendously. 

\begin{lemma}\label{lemma:two cases}
There exists an absolute constant $c\in (0, 10^{-4})$ such that for every $m\times n$ separated matrix $M$ of rank $r$, average entry $p\in (0,0.9)$, the following holds. There exists a submatrix $M'\subseteq M$ of size $m/2\times n/2$ such that  either
\begin{itemize}
    \item[(1)] $p(M')\leq p-c \sqrt{p/r}$ and $q(M')\leq 4q(M)$, or
    \item[(2)] $p(M')\leq p+12c\sqrt{p/r}$ and $q(M')\leq 2^{-100} q(M)$.
\end{itemize}
\end{lemma}
\begin{proof}
Set $\alpha=2^{-100}$ and $c=\min\{\alpha c_0/30, 10^{-4}\}$, where $c_0$ is the constant coming from Lemma~\ref{lemma:disc}.

First of all, observe that if the discrepancy of the matrix $M$ is at least $\disc(M)\geq 3 c \sqrt{p/r} mn$, we have outcome $(1)$. Namely, by Claim~\ref{claim:half}, we have a half-sized submatrix $M'\subseteq M$ with $p(M') \leq p-\disc(M)/3mn\leq p-c\sqrt{p/r}$.
Moreover, $q(M')\leq 4q(M)$ for every half-sized submatrix $M'$ simply because $q(M')\leq \frac{\|M'-pJ\|_F^2}{m/2\cdot n/2}\leq 4\frac{\|M-pJ\|_F^2}{mn}=4q(M)$, where the first inequality comes from Claim~\ref{claim:variance}.

If the discrepancy of $M$ is at most $\disc(M)\leq 3 c \sqrt{p/r} mn$, we get outcome $(2)$. Lemma~\ref{lemma:disc} implies that there exists a half-sized submatrix $M'\subseteq M$ for which $\disc(M)\geq c_0 mn\frac{q(M')}{\sqrt{rq(M)}}$. Combining the two inequalities and using that $q(M)\geq p /100$ (which comes from Lemma~\ref{lemma:large_F}), we find
\[ 3 c  mn\sqrt{\frac{p}{r}} \geq \disc(M) \geq c_0 mn\frac{q(M')}{\sqrt{rq(M)}} = c_0 mn \sqrt{\frac{q(M)}{r}} \cdot \frac{q(M')}{q(M)}\geq \frac{c_0}{10} mn \sqrt{\frac{p}{r}} \cdot \frac{q(M')}{q(M)}.\]

Cancelling out $mn\sqrt{p/r}$ and recalling that $c\leq \alpha c_0/30$, we get $\alpha \geq q(M')/q(M)$, i.e. $q(M')\leq 2^{-100}q(M)$, Finally, since $\disc(M)\leq 3c\sqrt{p/r} mn$, we have $p(M')-p(M)\leq \frac{3cmn\sqrt{p/r}}{mn/4}=12c\sqrt{p/r}$ by Claim \ref{claim:half}. This completes the proof.
\end{proof}

In our final preparatory lemma, we show how $M$ can be regularized so that no entry of $M$ is too large. This ensures that the variance of the submatrices of $M$ cannot be too large either. Also, a similar argument shows that if the density of a separated matrix is small enough, then there is a half-sized all-zero submatrix.

\begin{lemma}\label{lemma:low_average}
Let $M$ be an $m\times n$ real matrix of rank $r$.
\begin{itemize}
    \item[(1)] If every entry of $M$ is non-negative, then $M$ contains a submatrix of size at least $0.9m\times 0.9n$ with every entry less than $400r^2 p(M)$.
    \item[(2)] If at most $mn/(16r)$ entries of $M$ are non-zero, then $M$ contains an all-zero submatrix of size~$m/2\times n/2$.
\end{itemize}
\end{lemma}
\begin{proof}
(1) Let $p=p(M)$ and $\Delta=400r^2p$. Consider the bipartite graph $G$ whose vertex classes are $A=[m]$ and $B=[n]$, and $i\in A$ and $j\in B$ are joined by an edge if $M(i,j)\geq 400rp$. Then $G$ has at most $mn/(400r)$ edges. Remove all vertices of $A$ of degree more than $n/(20r)$ and all vertices of $B$ of degree more than $m/(20r)$, and let $G'$ be the resulting bipartite graph on vertex classes $A'\subset A$ and $B'\subset B$. Then $|A'|\geq \frac{19}{20}m,|B'|\geq \frac{19}{20}n$. Color each edge $(i,j)\in A'\times B'$ of $G'$ red if $M(i,j)\geq \Delta$.

Let $F$ be a maximal induced matching of $G'$ that consists only of red edges. We show that $|F|\leq r$. Otherwise, assume that $|F|\geq r+1$, and let $a_1,\dots,a_{r+1}\in A'$ and $b_1,\dots,b_{r+1}\in B'$ such that $\{a_i,b_i\}\in F$. After possibly permuting the rows and columns of $M$, the submatrix of $M$ spanned by rows  $\{a_1,\dots,a_{r+1}\}$ and columns $\{b_1,\dots,b_{r+1}\}$ is an $(r+1)\times (r+1)$ matrix in which the diagonal entries are at least $\Delta$, while the non-diagonal entries are less than $400rp=\Delta/r$. Such a matrix has full rank, for example by the Gershgorin circle theorem. This contradicts that $M$ has rank $r$.

Now remove every vertex of $G'$ which is connected by an edge to some vertex of $F$. Then we removed at most $|F|\cdot n/(20r)\leq n/20$ vertices from $B'$ and similarly at most $m/20$ vertices from $A'$, so the resulting bipartite graph $G''$ on vertex classes $A''$ and $B''$ satisfies $|A''|\geq 0.9m, |B''|\geq 0.9n$. Furthermore, $G''$ contains no red edges, otherwise we run into a contradiction with the maximality of $F$. In conclusion, $M[A''\times B'']$ is submatrix of $M$ of size at least $0.9m\times 0.9n$ with every entry less than~$\Delta$. 

(2) The proof of this part proceeds similarly. However, we define the graph $G$ by connecting $i$ and $j$ if $M(i,j)\neq 0$, then $G$ has at most $mn/(16r)$ edges. We find $A'\subset [m]$ and $B'\subset [n]$ such that $|A'|\geq 3m/4,|B'|\geq 3n/4$, and no vertex in $A'$ has degree more than $n/(4r)$ and no vertex of $B'$ has degree more than $m/(4r)$. Let $G'$ be the subgraph of $G$ induced by $A'\cup B'$.

Let $F'$ be the maximal induced matching in $G'$. Then $|F'|\leq r$, as $F'$ corresponds to a diagonal submatrix of $M$ with non-zero entries on the diagonal. Removing all neighbours of the vertices of $F'$, we remove at most $m/4$ vertices $A'$ and $n/4$ vertices from $B'$. Hence, denoting the remaining sets by $A''$ and $B''$, we have $|A''|\geq m/2$, $|B''|\geq n/2$, and $M[A''\times B'']$ is an all-zero submatrix.
\end{proof}

Finally we are ready to prove Theorem \ref{thm:sparse_matrix}, which we restate here for the reader's convenience.

\begin{theorem}\label{thm:average}
Let $M\in\mathbb{R}^{m\times n}$  be a rank $r$ matrix whose every entry is 0 or at least 1. If $p(M)=p \leq 1/2$, then $M$ contains an all-zero submatrix of size $n2^{-O(\sqrt{pr}+\log r)}$. 
\end{theorem}
\begin{proof}
We may assume that $r$ is sufficiently large. By Lemma \ref{lemma:low_average}, we find a square submatrix $M_0$ of size at least $0.9 m\times 0.9n$ such that every entry of $M_0$ is at most $400r^2p\leq 400r^2=:\Delta$. The average entry of $p_0=p(M_0)$ satisfies $p_0\leq p/0.9^2\leq 1.24 p<0.7$, since $p\leq 1/2$. Also, writing $q_0=q(M_0)$, we have $q_0\leq 10^6r^4$.

We now define a sequence of matrices $M_0,M_1,\dots, M_t$ such that $M_{i+1}$ is a half-sized submatrix of $M_{i}$ for $i=0, 1, \dots, t$. As long as $p(M_i) \in (\frac{1}{16r}, 0.8)$, we apply Lemma~\ref{lemma:two cases} to $M_i$, and choose $M_{i+1}$ to be a half-sized submatrix which satisfies one of the outcomes of the lemma. We stop the process at the first index $t$ for which $p_t\leq \frac{1}{16r}$ or $p_t\geq 0.8$. 

To sum up, writing $p_i=p(M_i)$ and $q_i=q(M_i)$, for $i=0, 1, \dots, t-1$, we get one of the following two conclusions 
\begin{itemize}
    \item[(1)] $p_{i+1}\leq p_i-c\sqrt{p_i/r}$ and $q_{i+1}\leq 4q_i$, or
    \item[(2)] $p_{i+1}\leq p_i+12c \sqrt{p_i/r}$ and $q_{i+1}\leq 2^{-100}q_i$.
\end{itemize}
 Here, $c\in (0, 10^{-4})$ denotes the absolute constant coming from Lemma~\ref{lemma:two cases}.

Assume for a second that the process stops because $p_t\leq \frac{1}{16r}$ at some point.
The matrix $M_t$ has size $m_t\times n_t$, where $m_t\geq 2^{-t-1}m, n_t\geq 2^{-t-1}n$, and at most $m_tn_t/(16r)$ non-zero entries. By Lemma \ref{lemma:low_average}, $M_t$ contains an all-zero submatrix of size at least $2^{-t-2}m\times 2^{-t-2}n$. Therefore, in order to prove the theorem, we need to show that the process stops after fewer than $t=O(\sqrt{pr}+\log r)$ steps with $p_t\leq 1/(16 r)$.

Note that $p_i\leq 0.9$ for all $i\leq t$. For $i\leq t-1$, this is obvious since $p_i\leq 0.8$, but we also have $p_t\leq p_{t-1}+12c\leq 0.8+0.1$, since $c\leq 10^{-4}$. Hence, if $q_i\leq \frac{1}{10^4 r}$ for some $i$, then $p_i\leq 100q_i\leq \frac{1}{100r}$, showing that $i=t$, i.e that the process stopped with $p_t\leq 1/(16r)$.

In order to analyze the process, we define the progress function \[f(p_i, q_i)=\sqrt{r p_i}+\frac{c}{10}\log_2 (\Delta q_i).\] First of all, for all $i<t$, we have $f(p_i, q_i)\geq 0$, since $p_i\geq 0$ and $q_i\geq \frac{1}{10^4 r}$. Also, $f(p_0, q_0)\leq O(\sqrt{pr}+\log r)$, since $q_0\leq \Delta^2$. Finally, the key property of the function $f$ is that at every step of the procedure it is decreased by an absolute constant. This means that within $O(\sqrt{pr}+\log r)$ steps, one gets to the situation where $p_i\leq \frac{1}{16 r}$, thus stopping the iteration.

We now prove the key property of $f$, i.e. that for all $0\leq i\leq t-1$, we have \[f(p_{i+1}, q_{i+1})\leq f(p_i, q_i)-\frac{c}{4}.\]
Distinguish two cases, based on the outcome of Lemma~\ref{lemma:two cases}. In case of outcome (1), we have $rp_{i+1}\leq rp_i-c\sqrt{rp_i}$ and $q_{i+1}\leq 4q_i$, thus giving
\[f(p_{i+1}, q_{i+1})\leq \sqrt{rp_i-c\sqrt{rp_i}}+\frac{c}{10}\log_2(4\Delta q_i)\leq \sqrt{rp_i}-\frac{c}{2}+\frac{c}{10}\log_2(\Delta q_i)+\frac{c}{5}\leq f(p_i, q_i)-\frac{c}{4}.\]
In the second inequality, we used the general inequality $\sqrt{a+\delta\sqrt{a}}\leq \sqrt{a}+\delta/2$, which holds for all $\delta\geq -\sqrt{a}$, and can be proven by squaring both sides.

In case of outcome (2), we have $rp_{i+1}\leq rp_i+12 c\sqrt{rp_i}$ and $q_{i+1}\leq 2^{-100}q_i$, thus giving
\[f(p_{i+1}, q_{i+1})\leq \sqrt{rp_i+12c\sqrt{rp_i}}+\frac{c}{10}\log_2(2^{-100}\Delta q_i)\leq \sqrt{rp_i}+6c+\frac{c}{10}\log_2(\Delta q_i)-10c\leq f(p_i, q_i)-\frac{c}{4}.\]

The proof now follows easily. Since $f(p_t, q_t)\leq f(p_0, q_0)\leq \sqrt{rp_0}+\frac{c}{10}\log(10^6r^4)$, we have $\sqrt{rp_t}\leq \sqrt{0.7 r}+O(\log r)$. This also shows that $p_t\geq 0.8$ cannot happen, using our assumption that $r$ is sufficiently large. Furthermore, since $0\leq f(p_i, q_i)\leq f(p_0, q_0)-ci/4$ for all $i\leq t-1$ and $f(p_0, q_0)=O(\sqrt{pr}+\log r)$, the process must stop after at most $\frac{4}{c}\cdot f(p_0, q_0)=O(\sqrt{pr}+\log r)$ steps. This finishes the proof.
\end{proof}

\section{Constructions} \label{sect:construction}

In this section, we prove the claimed properties of the constructions discussed in the Introduction. Recall that we defined the $n\times n$ matrix $M_2$ as the intersection matrix of $(\cF, \cF)$, where $\cF=\binom{[r]}{k}$ is the family of all $k$-element subsets of $[r]$ and $k=\sqrt{\eps r}, n=\binom{r}{k}$.

\begin{lemma}
The matrix $M_2$ satisfies the following properties.
\begin{itemize}
    \item[1.] $M_2$ has $\Theta(\eps n^2)$ non-zero entries.
    \item[2.] $M_2$ contains no square all-zero submatrix of size larger than $2^{-\sqrt{\eps r}}n$.
    \item[3.] The average of the entries of $M_2$ is $\eps$.
\end{itemize}
\end{lemma}
\begin{proof}
1. For every $A\in \binom{[r]}{k}$, the number of sets $B\in \binom{[r]}{k}$ disjoint from $A$ is 
$$\binom{r-k}{k}=\binom{r}{k}\frac{r-k}{r}\dots \frac{r-2k+1}{r-k}=n\cdot \left(1-\frac{k}{r}\right)\dots\left(1-\frac{k}{r-k}\right)=ne^{-\Theta(k^2/r)}.$$
Here, $k^2/r=\eps$, so we have $ne^{-\Theta(k^2/r)}=ne^{-\Theta(\eps)}=(1-\Theta(\eps)) n$. In conclusion, the total number non-zero entries of $M$ is $\Theta(\eps n^2)$.

\noindent
2. Note that if an all-zero submatrix has rows and columns indexed by families $\cA,\cB\subseteq \binom{[r]}{k}$, then $\cA$ and $\cB$ must  be cross-disjoint. That is, there exists a partition $[r]=U\cup V$ such that all sets of $\cA$ are contained in $U$, and all sets of $\cB$ are contained in $V$. Since one of $U$ and $V$ is at most $r/2$, say $|U|\leq r/2$, we have $|\cA|\leq \binom{r/2}{k}\leq \frac{1}{2^{k}}\binom{r}{k}=2^{-\sqrt{\eps r}}n$. Therefore, there does not exist a square all-zero submatrix of $M$ of size larger than $2^{-\sqrt{\eps r}}n$.

\noindent
3. As every row of $M$ is identical up to permutation, it is enough to calculate the average size of the intersection $|A\cap B|$ for a fixed set $A\in \binom{[r]}{k}$. This is equal to the expected value of $|A\cap B|$, if $B$ is chosen randomly from the uniform distribution on $\binom{[r]}{k}$. Since for every $i\in [r]$ we have $\mathbb{P}[i\in B]=\frac{k}{r}$, we conclude $\mathbb{E}\big[|A\cap B|\big]=\sum_{i\in A}\Pb[i\in B]=k\cdot \frac{k}{r}=\eps$.
\end{proof}

Recall now the matrix $M_3$, which is defined in the Introduction as $M_3=N-\frac{r}{4}J$, where $N$ is the intersection matrix of $(2^{[r]}, 2^{[r]})$.

\begin{lemma}
The matrix $M_3$ contains $\Omega(n^2/\sqrt{r})$ zero entries, but it contains no all-zero submatrix of size $s\times t$ if $st> 2^{r}$.
\end{lemma}
\begin{proof}
The number of zero entries of the matrix $M_3$ equals to the number of pairs of sets $(A, B)$ with $|A\cap B|=r/4$ and $A, B\subseteq [r]$. The number of such pairs is $\binom{r}{r/4}3^{3r/4}$, since there are $\binom{r}{r/4}$ ways to choose the intersection $A\cap B$, and for each element outside the intersection there are three choices, for this element can belong to $A$, $B$ or neither of the sets. A simple calculation using the Stirling approximation, which states that $k!=\Theta\big(\sqrt{k}\big(\frac{k}{e}\big)^k\big)$, gives the following
\[\binom{r}{r/4}3^{3r/4}=\Theta\left(\frac{\sqrt{r}\big(\frac{r}{e}\big)^r}{\sqrt{r}\big(\frac{r}{4e}\big)^{r/4}\cdot \sqrt{r}\big(\frac{3r}{4e}\big)^{3r/4}}3^{3r/4}\right)=\Theta\left(\frac{1}{\sqrt{r}\big(\frac{1}{4}\big)^{r/4}\cdot \big(\frac{3}{4}\big)^{3r/4}}3^{3r/4}\right)=\Theta\left(\frac{4^r}{\sqrt{r}}\right).\]
Hence, $M_3$ has $\Theta(n^2/\sqrt{r})$ zero entries, as claimed.

Furthermore, an all-zero submatrix of size $s\times t$ corresponds to a collection of $s+t$ sets $A_1, \dots, A_s, B_1, \dots, B_t\in  2^{[r]}$ such that $|A_i\cap B_j|=r/4$ for all $i, j$. A theorem of Sgall \cite{Sgall} states that if two families $\cA,\cB\subseteq 2^{[r]}$ have the property that $|A\cap B|$ is the same for all $A\in \cA, B\in \cB$, then  $|\cA||\cB|\leq 2^{r}$ (see Corollary 3.5 of \cite{Sgall}). This finishes the proof.
\end{proof}

\section{Open problems}\label{sect:open_problems}

In Section~\ref{sect:daykin-erdos}, we proposed to study a variant of the Daykin-Erd\H os problem for families with quadratically many pairs of sets with intersection $\lambda\neq 0$. We showed that two set families $\cA, \cB\subseteq 2^{[n]}$ of size $m$ with $cm^2$ disjoint pairs can have size at most $m\leq 2^{(1+o(1))n/2}$. However, our arguments only work when the density $c$ is a constant. Therefore, it would be interesting to understand the dependency of the maximal size of the families $\cA, \cB$ when the density $c$ is in the range $2^{-n}\leq c\leq 1$. 

\begin{problem}
Let $\cA, \cB\subseteq 2^{[n]}$ be two set families of size $m$ and let $\lambda\in [n]$, and $c\in [2^{-n}, 1]$ be parameters. Assuming there are at least $cm^2$ pairs of sets $(A, B)\in \cA\times \cB$ such that $|A\cap B|=\lambda$, what is the largest possible size of $|\cA| |\cB|$, as a function of $c$ and $n$?
\end{problem}

Furthermore, in Section~\ref{sect:singer-sudan} we showed that an intersection matrix of two set systems with many disjoint pairs must contain a large all-zero submatrix. It is natural to ask whether the same holds for arbitrary nonzero intersections as well.

\begin{problem}\label{prob:intermediate}
Let $\cA, \cB\subseteq 2^{[n]}$ be two set systems of size $m$ and at least $(1-\eps) m^2$ pairs with intersection $\lambda$. Do there exist subfamilies $\cA'\subseteq \cA$ and $\cB'\subseteq \cB$, of size $|\cA'||\cB'|\geq 2^{-O(\sqrt{n\eps})}|\cA||\cB|$ such that each pair $(A', B')\in \cA'\times \cB'$ has intersection $\lambda$?
\end{problem}

Note that Problem~\ref{prob:intermediate} is still a special case of Conjecture~\ref{conj:lovett}. Namely, if $M$ is the intersection matrix of set systems $\cA, \cB$, then the matrix $M-\lambda J$ contains $(1-\eps)m^2$ zero entries and has rank at most $n+1$, so it should contain a large all-zero submatrix. Nevertheless, we believe that it may be easier to tackle Problem~\ref{prob:intermediate} than the case of general matrices, because intersection matrices of set systems are more structured than general low-rank matrices.

\bigskip
\noindent
\textbf{Acknowledgement.} We would like to especially thank Matthew Kwan and Lisa Sauermann for sharing their elegant proof of Lemma~\ref{lemma:cut} with us.

\newpage

\appendix
\section{Appendix}

In this appendix, we give a proof of our final result mentioned in the Introduction, namely that matrices with small integer entries also have large monochromatic rectangles. The motivation for this theorem, as for many others in this paper, is provided by the log-rank conjecture. Recall that the recent work of the last two authors shows that $m\times n$ binary matrices of rank $r$ contain constant submatrices of size $2^{-O(\sqrt{r})}m\times 2^{-O(\sqrt{r})}n$. However, it seems that the assumption that the matrix is binary is not crucial - the same holds for matrices with entries $\{0, \dots, t\}$, where $t$ is some constant.
More precisely, our goal is prove Theorem~\ref{thm:few_distinct_values}, which states that a rank $r$ matrix $M\in\mathbb{N}^{m\times n}$ with $p(M)\leq t$ contains a constant submatrix of size $2^{-O(t\sqrt{r})}m\times 2^{-O(t\sqrt{r})}n$.

\medskip

Before we discuss the proof of Theorem~\ref{thm:few_distinct_values}, we mention that this problem changes significantly depending on the relationship between $t$ and $r$. For example, if $t\gg r^{1+c}$ one can understand the size of largest constant rectangles in matrices with only $t$ distinct entries using results about point-hyperplane incidences. Namely, if $M$ is an $m\times n$ rank $r$ matrix with $t$ distinct entries, then there is $\lambda$ such that $M-\lambda J$ contains at least $mn/t$ zero entries. The result of Singer and Sudan \cite{SS22} discussed in the Introduction states that a rank $r$ matrix with $\delta mn$ zero entries contains an all-zero submatrix with at least $\Omega(\delta^{2r}mn/r)$ entries, which in our case means that $M$ contains a constant-$\lambda$ submatrix of with at least $\Omega((1/t)^{2r} mn/r)=2^{-O(r\log t)}mn$ entries since $\delta=1/t$ (see also \cite{MST} for somewhat sharper bounds in case $t$ is large). In case $t\gg r^{1+c}$ for some constant $c>0$, this bound is sharp as the following construction demonstrates. 

\medskip
\noindent
\textbf{Construction 4.} Let $k$ be an integer and let $n=(2k+1)^r$. Let $v_1,\dots,v_{n}$ be an enumeration of the vectors in $\{-k,\dots,k\}^r$, and let $M_4$ be the $n\times n$ matrix defined as $M_4(i,j)=\langle v_i,v_j\rangle$. 
\medskip

Then $\rank(M_4)\leq r$, and we show in the next lemma that $M_4$ contains no constant square submatrix of size larger than $(2k+1)^{r/2}$. In case $t\gg r^{1+c}$, choosing  $k = \frac{1}{2}\sqrt{t/r}\gg t^{c/2(1+c)}$, one has $(2k+1)^{r/2}=2^{-\Omega(r\log t)}n$. Also, if $k = \frac{1}{2}\sqrt{t/r}$, there are at most $2 k^2 r+1<t$ distinct entries in the matrix $M_4$.

\begin{lemma}
The matrix $M_4$ contains no constant square submatrix of size larger that $(2k+1)^{r/2}$.
\end{lemma}
\begin{proof}
Let $R$ and $C$ be sets of rows and columns such that $M[R\times C]$ is a constant submatrix. Let $U$ be the affine span of $\{v_i:i \in R\}$ and $V$ be the affine span of $\{v_j: j\in C\}$. Then, $U$ and $V$ are orthogonal affine subspaces of $\mathbb{R}^r$ (we say that two affine spaces are orthogonal if translating them to the origin gives a pair of orthogonal linear spaces. Hence, we have $\dim(U)+\dim(V)\leq r$, showing that one of them, say $U$, has dimension at most $\dim(U)=d\leq r/2$. Next, we argue that a space of dimension $d$ cannot intersect the grid $\{-k, \dots, k\}^r$ in more than $(2k+1)^d$ points. Let $f_1, \dots, f_d$ be a basis of $U$ in reduced-row echelon form, i.e. there is a set $I\subset [d]$ such that $f_1,\dots,f_d$ are the $d$ standard basis vectors when restricted to the coordinates in $I$. Such a basis always exists by the Gauss-Jordan elimination. Looking at the coordinates in $I$, it is clear that the only linear combinations $\lambda_1f_1+\dots+\lambda_df_d$ which can lie in the grid $\{-k, \dots, k\}^r$ are those where $\lambda_1, \dots, \lambda_d\in \{-k, \dots, k\}$. Thus, there are at most $(2k+1)^d$ linear combinations that fall into the grid $\{-k,\dots,k\}^{r}$, so $|U\cap\{-k,\dots,k\}^{r}|\leq (2k+1)^d$. Hence, the largest constant square submatrix of $M$ has size at most $(2k+1)^{r/2}$.
\end{proof}

Let us now turn our attention to Theorem~\ref{thm:few_distinct_values}, in which we show that when $t\ll \sqrt{r}$, the upper bounds mentioned above can be greatly improved. We now restate this theorem for the reader's convenience.

\begin{theorem}\label{thm:few_distinct_values_appendix}
Let $M$ be an $m\times n$ matrix of rank $r$ with nonnegative integer entries. If  $p(M)\leq t$ for some $t\geq 1$, then $M$ contains a constant submatrix of size at least $2^{-O(t\sqrt{r})}m \times 2^{-O(t\sqrt{r})} n.$
\end{theorem}

The basic idea of the proof is similar to the one we encountered in Section~\ref{sect:low-rank matrices} - we show that matrices in question have high discrepancy and we use this to decrease the average entry until a greedy argument can be used to find an all-zero submatrix. However, the major new challenge which we did not encounter in the proof of Theorem~\ref{thm:sparse_matrix} is to lower-bound the discrepancy of matrices whose average value may be close to an integer. In contrast, in the proof of Theorem~\ref{thm:sparse_matrix}, it was of crucial importance that the average entry of the matrix was bounded away from $1$ in order to establish good lower bounds on variance and discrepancy. 

Therefore, we begin by generalizing several lemmas from Section~\ref{sect:low-rank matrices} to matrices whose average entry may not be bounded away from $1$. The following lemma extends Lemma~\ref{lemma:large_F}.

\begin{lemma}\label{lemma:large_F_appendix}
Let $M$ be a separated $m\times n$ matrix. Then $q(M)\geq {p(M)(1-p(M))}/{100}$.
\end{lemma}
\begin{proof}
Note that the statement is only nontrivial when $p(M)\in (0, 1)$, since otherwise the right-hand side is negative, while $q(M)\geq 0$ always holds. Furthermore, recall that we proved in Lemma~\ref{lemma:large_F} that a separated $m\times n$ matrix $M$ with $p(M)\in (0, 0.9)$ satisfies $q(M)\geq p(M)/100\geq {p(M)(1-p(M))}/{100}$, so in this case we are done. On the other hand, if $p(M)\in (0.9, 1)$, then $p(J-M)=1-p(M)\in (0, 0.1)$ and therefore we have $q(M)=q(J-M)\geq (1-p(M))/100\geq p(M)(1-p(M))/100$.
\end{proof}

The following result is a slight variation of Lemma~\ref{lemma:two cases} and, as in Section~\ref{sect:low-rank matrices}, this lemma  represents the engine behind the iteration which decreases the average entry of a matrix. The proof is essentially the same as proof of Lemma~\ref{lemma:two cases}, the only changes being that $2^{-100}$ is now $2^{-200}$ and we use Lemma~\ref{lemma:large_F_appendix} to lower-bound the variance of $M$ instead of Lemma~\ref{lemma:large_F}. Therefore, we omit the proof.

\begin{lemma}\label{lemma:two cases appendix}
There exists an absolute constant $c\in (0, 10^{-4})$ such that for every $m\times n$ separated matrix $M$ of rank $r$, we have the following. Let $p=p(M)\in (0,1)$, then there exists a submatrix $M'\subseteq M$ of size $m/2\times n/2$ such that one of the two alternatives hold
\begin{itemize}
    \item[(1)] $p(M')\leq p-c \sqrt{\frac{p(1-p)}{r}}$ and $q(M')\leq 4q(M)$, or
    \item[(2)] $p(M')\leq p+12c\sqrt{\frac{p(1-p)}{r}}$ and $q(M')\leq 2^{-200} q(M)$.
\end{itemize}
\end{lemma}

The following proposition generalizes the idea of the proof of Theorem~\ref{thm:sparse_matrix}.

\begin{proposition}\label{prop:iteration}
Let $M\in\mathbb{R}^{m\times n}$ be a rank $r$ separated matrix with entries in $(-r^3, r^3)$. If $p(M)=p \leq 1-\frac{2}{\sqrt{r}}$, then $M$ contains a submatrix $M'\subseteq M$ of size $m2^{-O(\sqrt{r})}\times n2^{-O(\sqrt{r})}$ satisfying $p(M')\leq \frac{1}{16r}$.
\end{proposition}
\begin{proof}
We may assume that $r$ is sufficiently large.
Let $p_0=p(M)$ and $q_0=q(M)$, for which we have $q_0\leq r^6$ since all entries of the matrix are bounded by $r^3$ in absolute value.

We define a sequence of matrices $M_0, M_1, \dots M_s$ such that $M_{i+1}$ is a half-sized submatrix of $M_{i}$ for $i=0, 1, \dots, s$. As long as $p(M_i) \in (\frac{1}{16r}, 1-\frac{1}{\sqrt{r}})$, we apply Lemma~\ref{lemma:two cases appendix} to $M_i$, and choose $M_{i+1}$ to be a half-sized submatrix which satisfies one of the outcomes of the lemma. We stop the process at the first index $s$ for which $p_s\leq \frac{1}{16r}$ or $p_s\geq 1-\frac{1}{\sqrt{r}}$. 
To sum up, writing $p_i=p(M_i)$ and $q_i=q(M_i)$, as long as $p_i\in (\frac{1}{16r}, 1-\frac{1}{\sqrt{r}})$, we get one of the following two conclusions 
\begin{itemize}
    \item[(1)] $p_{i+1}\leq p_i-c\sqrt{\frac{p_i(1-p_i)}{r}}$ and $q_{i+1}\leq 4q_i$, or
    \item[(2)] $p_{i+1}\leq p_i+12c \sqrt{\frac{p_i(1-p_i)}{r}}$ and $q_{i+1}\leq 2^{-200}q_i$.
\end{itemize}
\noindent
Here, $c\in (0, 10^{-4})$ is the absolute constant coming from Lemma~\ref{lemma:two cases appendix}.

 The matrix $M_s$ has size $m_s\times n_s$, where $m_s\geq 2^{-s-1}m, n_s\geq 2^{-s-1}n$. Therefore, in order to prove the theorem, we need to show that the process stops after fewer than $s=O(\sqrt{r}+\log r)$ steps and $p_s\leq 1/(16 r)$. Note that $p_i\leq 1-\frac{1}{2\sqrt{r}}$ for all $i\leq s$. For $i\leq s-1$, this is obvious since $p_i\leq 1-\frac{1}{\sqrt{r}}$, but we also have $p_s\leq p_{s-1}+12c\sqrt{\frac{p_{s-1}(1-p_{s-1})}{r}}\leq 1-\frac{1}{\sqrt{r}}+\frac{1}{2\sqrt{r}}$, where we have used that $p_i(1-p_i)\leq 1$, $c\leq 10^{-4}$. Note that if we ever have $q_{i}\leq \frac{1}{10^4 r}$, then $p_i(1-p_i)\leq 100q_i\leq \frac{1}{32 r}$ by Lemma~\ref{lemma:large_F_appendix}. This implies $p_i\leq \frac{1}{16 r}$ or $p_i\geq 1-\frac{1}{16r}$. However, it cannot be $p_i\geq 1-\frac{1}{16r}$, since the process always stops at a point $p_i\leq 1-\frac{1}{2\sqrt{r}}<1-\frac{1}{16r}$. Hence, we may assume that  $q_i\geq \frac{1}{10^4 r}$ for every $i<s$, as otherwise we immediately get $p_i\leq \frac{1}{16r}$ as well. We write $\Delta=10^4r$.

In order to analyze our process, we define progress functions \[f(p)=\sqrt{r p}+\sqrt{r}-\sqrt{r (1-p)}\text{\ \ \ and \ \ }g(q)=\frac{c}{10}\log_2 (\Delta q).\]
First of all, for all $i<s$, we have $f(p_i), g(q_i)\geq 0$, since $p_i\in (0, 1)$ and $\Delta q_i\geq 1$. Also, $f(p_0)\leq O(\sqrt{r})$. Furthermore, $g(q_0)\leq O(\log r)$ since $\Delta q_0$ is at most polynomial in $r$. So, $f(p_0)+g(q_0)\leq O(\sqrt{r}+\log r)$.
Finally, the key property of the function $f$ is that at every step of the procedure it is decreased by an absolute constant.

\begin{claim}\label{claim:progress function}
For all $i\leq s-1$, \[f(p_{i+1})+g(q_{i+1})\leq f(p_{i})+ g(q_{i})-\frac{c}{20}.\]
\end{claim}
\begin{proof}
If $i\leq s-1$, we have either outcome $(1)$ or $(2)$ of Lemma~\ref{lemma:two cases appendix}. We now look at the behavior of $f(p)+g(q)$ in both of these cases. Note that $f(p)$ is increasing in $p$. If outcome $(1)$ occurs, we have $p_{i+1}\leq x$, where $x=p_i-c\sqrt{\frac{p_i(1-p_i)}{r}}$, and therefore by considering the Taylor expansion of $f$,
\begin{equation}\label{eqn:derivative bound 1}
    f(p_{i+1}) \leq f(x)\leq f(p_i)-|p_i-x|\min_{x\leq u\leq p_i}f'(u).
\end{equation}

We can compute the derivative $f'(u)=\frac{\sqrt{r}}{2\sqrt{u}}+\frac{\sqrt{r}}{2\sqrt{1-u}}\geq \frac{\sqrt{r}}{2\sqrt{2u(1-u)}}$, where we used that either $u\geq 1/2$ or $1-u\geq 1/2$. 
Note that for $p_i\leq 1-\frac{1}{\sqrt{r}}\leq 1-\frac{1}{r}$, we have $c\sqrt{\frac{p_i(1-p_i)}{r}}\leq \frac{1-p_i}{25}$, which can be proven by squaring the equation and observing $cp_i(1-p_i)\leq 10^{-4} (1-p_i)\leq \frac{(1-p_i)^2r}{25^2}$. Hence, $1-u\leq 1-x=1-p_i+c\sqrt{\frac{p_i(1-p_i)}{r}}\leq 2(1-p_i)$. 
It follow that for all $u\in [x, p_i]$, we have $f'(u)\geq \frac{\sqrt{r}}{2\sqrt{2u(1-u)}}\geq \frac{\sqrt{r}}{4\sqrt{p_i(1-p_i)}}$. Plugging this into (\ref{eqn:derivative bound 1}) gives us
\[f(p_{i+1}) \leq f(x)\leq f(p_i)-c\sqrt{\frac{p_i(1-p_i)}{r}}\cdot \frac{\sqrt{r}}{4\sqrt{p_i(1-p_i)}}=f(p_i)-\frac{c}{4}.\]
We also have $g(q_{i+1})\leq \frac{c}{10}\log_{2}(4\Delta q_i)=\frac{c}{10}\log_{2}(\Delta q_i)+\frac{c}{5}=g(q_i)+\frac{c}{5}$, thus showing that in case of outcome (1) we have
\[f(p_{i+1})+g(p_{i+1})\leq f(p_i)+g(q_i)-c/20.\]

In case of outcome $(2)$, we let $x=p_i+12c \sqrt{\frac{p_i(1-p_i)}{r}}$. Then, $f(p_{i+1})\leq f(x)\leq f(p_i)+|x-p_i|\max_{p_i\leq u\leq x} f'(u)$.
Similarly as before, we have $f'(u)= \frac{\sqrt{r}}{2\sqrt{u}}+\frac{\sqrt{r}}{2\sqrt{1-u}}\leq \frac{\sqrt{r}}{\sqrt{u(1-u)}}$. Since $12c\sqrt{\frac{p_i(1-p_i)}{r}}\leq \frac{12}{25}(1-p_i)\leq \frac{1-p_i}{2}$, for all $u\in [p_i, x]$ it holds that $1-u\geq 1-x\geq (1-p_i)/2$. Hence, $f'(u)\leq \frac{\sqrt{r}}{\sqrt{u(1-u)}}\leq \frac{\sqrt{2r}}{\sqrt{p_i(1-p_i)}}$ and so
\[f(p_{i+1}) \leq f(x) \leq f(p_i)+12c\sqrt{\frac{p_i(1-p_i)}{r}}\cdot \frac{\sqrt{2r}}{\sqrt{p_i(1-p_i)}}=f(p_i)+12\sqrt{2}c.\]
We also have $g(q_{i+1})\leq g(2^{-200}q_i)=g(q_i)-200\frac{c}{10}$, thus showing that \[f(p_{i+1})+g(p_{i+1})\leq f(p_i)+g(q_i)-20c+12\sqrt{2}c\leq f(p_i)+g(q_i)-c/20,\] since $12\sqrt{2}\leq 12\cdot 3/2=18$. This finishes the proof of Claim~\ref{claim:progress function}.
\end{proof}

The proof of Proposition~\ref{prop:iteration} now follows easily. First, we show that we cannot have $p_s\geq 1-\frac{1}{\sqrt{r}}$. If $p_s\geq 1-\frac{1}{\sqrt{r}}$, we also have $g(q_s)\geq 0$, and so 
\[f(p_s)\leq f(p_s)+g(q_s)\leq f(p_0)+g(q_0)\leq f(p_0)+\frac{c}{10}\log(\Delta \cdot r^6).\]
But $f(p_s)-f(p_0)\leq f(1-\frac{1}{\sqrt{r}})-f(1-\frac{2}{\sqrt{r}})=\sqrt{r-\sqrt{r}}-\sqrt[4]{r}-\sqrt{r-2\sqrt{r}}+\sqrt{2\sqrt{r}}\geq (\sqrt{2}-1)r^{1/4}> \frac{c}{10}\log (10^4 r^7)$ when $r$ is large. This shows that for $r$ sufficiently large we cannot have $p_s\geq 1-\frac{1}{\sqrt{r}}$.

Furthermore, since $0\leq f(p_i)+g(q_i)\leq O(\sqrt{r}+\log r)$ for all $0\leq i\leq s-1$ and $f(p_{i+1})+g(q_{i+1})\leq f(p_i)+g(q_i)-\frac{c}{20}$, we conclude that in at most $s=O(\sqrt{r})$ steps the process is terminated by reaching $p_s\leq 1/16r$. This completes the proof of Proposition~\ref{prop:iteration}.
\end{proof}

The following lemma shows a lower bound on $q(M)$ for integer matrices without half-sized constant submatrices. In some sense, it is the analogue of Lemma~\ref{lemma:large_F}, for integer matrices of arbitrary average entry.

\begin{lemma}\label{lemma:small_variance}
Let $M\in \mathbb{R}^{m\times n}$ be an integer matrix containing no half-sized constant submatrix. Then $q(M)\ge \frac{1}{128 r}$.
\end{lemma}
\begin{proof}
Let $t$ be the integer closest to $p=p(M)$. If at least $(1-\frac{1}{16(r+1)}) mn$ entries of $M$ are equal to $t$, then we apply Lemma~\ref{lemma:low_average} to $M-tJ$, which is a sparse matrix of rank at most $r+1$. Hence, we find an $m/2\times n/2$ submatrix of $M$ whose all entries equal $t$. 

Having assumed there is no such submatrix, we conclude there are at least $\frac{mn}{16(r+1)} $ entries of $M$ are different from $t$. For each such entry $M(i, j)\neq t$, we have $|M(i, j)-p|\geq 1/2$. Hence, $q(M)mn \ge \sum_{M(i,j)\neq t} (M(i, j)-p)^2\ge \frac{1}{4}\cdot \frac{1}{16(r+1)}mn$. We conclude that $q(M) \geq \frac{1}{128r}$ as claimed. 
\end{proof}

Our next lemma is a variant of Lemma~\ref{lemma:two cases appendix}, which allows us to take steps of size $\Theta(1/r)$ instead of $\Theta(\sqrt{p/r})$, which is useful when $p\leq O(1/r)$.

\begin{lemma}\label{lemma:two cases variant appendix}
There exists an absolute constant $c>0$ such that for every $m\times n$ integer matrix $M$ of rank $r$ with no constant submatrices of size $m/4\times n/4$, the following holds. There exists a submatrix $M'\subseteq M$ of size $m/2\times n/2$ with one of the following two alternatives:
\begin{itemize}
    \item[(1)] $p(M')\leq p(M)-c/r $ and $q(M')\leq 4q(M)$, or
    \item[(2)] $p(M')\leq p(M)+12c/r$ and $q(M')\leq 2^{-200} q(M)$.
\end{itemize}
\end{lemma}
\begin{proof}
The proof of this lemma is just a variant of the proof of Lemma~\ref{lemma:two cases}. Set $\alpha=2^{-200}$ and $c=\alpha c_0$, where $c_0$ is the constant coming from Lemma~\ref{lemma:disc}. Write $p=p(M)$.

If the discrepancy of the matrix $M$ is at least $\disc(M)\geq 3 c mn/r$, we have outcome $(1)$. Indeed, by Claim~\ref{claim:half}, we have a half-sized submatrix $M'\subseteq M$ with $p(M') \leq p-\disc(M)/3mn\leq p-c/r$.
Moreover, $q(M')\leq 4q(M)$ for every half-sized submatrix $M'$ as we have seen before.

If the discrepancy of $M$ is at most $\disc(M)\leq 3 c mn/r$, we get outcome $(2)$. Indeed, Lemma~\ref{lemma:disc} implies that there exists a submatrix $M'\subseteq M$ for which $\disc(M)\geq c_0 mn\frac{q(M')}{\sqrt{rq(M)}}$. Observe that if $q(M')\leq \frac{1}{10^4r}$, we must have a half-sized constant submatrix of $M'$ by the previous lemma, which is assumed not to be the case. Hence, $q(M')\geq \frac{1}{10^4r}$ and so 
\[\frac{3cmn}{r}\geq c_0mn\frac{q(M')}{\sqrt{rq(M)}}\geq \frac{c_0mn}{100r} \sqrt{\frac{q(M')}{q(M)}}.\]
Canceling and squaring, we get $q(M')\leq (300\alpha)^2q(M)\leq 2^{-200}q(M)$, showing that we find ourselves in outcome $(2)$. This completes the proof.
\end{proof}

Now, we are ready to prove Theorem~\ref{thm:few_distinct_values}.

\begin{proof}[Proof of Theorem~\ref{thm:few_distinct_values}.]
The proof hinges on the following key statement. If an integer-valued matrix $M$ has $p(M)\in [\ell-\frac{2}{\sqrt{r}}, \ell+1-\frac{2}{\sqrt{r}})$ for some integer $\ell$, then either there is submatrix $M'\subseteq M$ of size $m'\times n'$ with $\frac{m}{m'}, \frac{n}{n'}\leq 2^{O(\sqrt{r})}$ which is either constant or $p(M')\in [\ell-1-\frac{2}{\sqrt{r}}, \ell-\frac{2}{\sqrt{r}})$. 

Before showing this statement, let us argue why it is sufficient to complete the proof. By the assumption of the theorem, we have $p(M)\in [\ell-\frac{2}{\sqrt{r}}, \ell+1-\frac{2}{\sqrt{r}})$ for some $\ell\leq t$. Since no submatrix of $M$ can have negative average, after at most $t$ iterative applications of the key statement, we obtain a constant submatrix $M_t\subseteq M$ of size $m_t\times n_t$ where $\frac{m}{m_t}, \frac{n}{n_t}\leq 2^{O(t\sqrt{r})}$, thus completing the proof.

We now justify the key statement. First of all, note that there exists a submatrix $M_0\subseteq M$ of size $m_0\times n_0$ with $p(M_0)\leq \ell+\frac{1}{16r}$ and $\frac{m}{m_0}, \frac{n}{n_0}\leq 2^{O(\sqrt{r})}$. If $p(M)\leq \ell+\frac{1}{16r}$ this is trivial since we can take $M_0=M$. Otherwise, if $p(M)\in (\ell+\frac{1}{16r}, \ell+1-\frac{2}{\sqrt{r}})$, we can apply Proposition~\ref{prop:iteration} to the separated matrix $M-\ell J$ to obtain the submatrix $M_0\subseteq M$ with desired properties.

Our next goal is to find an $m_s\times n_s$ submatrix $M_s\subseteq M_0$ with average entry $p(M_s)\leq \ell-\frac{2}{\sqrt{r}}$ and  $\frac{m_0}{m_s}, \frac{n_0}{n_s}\leq 2^{O(\sqrt{r})}$. As in the proof of Proposition~\ref{prop:iteration}, we define a sequence of matrices $M_0\supset M_1\supset\dots\supset M_s$ as follows. The matrix $M_{i+1}$ is a half-sized submatrix of $M_{i}$ for $i=0, 1, \dots, s$, obtained by an application of Lemma~\ref{lemma:two cases variant appendix} to $M_i$. We stop the sequence as soon as either $M_i$ contains a quarter-sized constant submatrix or $p(M_i)\leq \ell-\frac{2}{\sqrt{r}}$. Writing $p_i=p(M_i)$ and $q_i=q(M_i)$, for each $i\leq s-1$ we have one of the two outcomes  
\begin{itemize}
    \item[(1)] $p_{i+1}\leq p_i-c/r$ and $q_{i+1}\leq 4q_i$, or
    \item[(2)] $p_{i+1}\leq p_i+12c/r$ and $q_{i+1}\leq 2^{-200}q_i$.
\end{itemize}

Defining the progress function $f(p_i, q_i)=rp_i+\frac{c}{10}\log_2(\Delta q_i)$, where $\Delta=10^4 r$, we see that $f(p_{i+1}, q_{i+1})\leq f(p_i, q_i)-\frac{c}{2}$ for each $i\leq s-1$. The justification is quite simple - in case of outcome $(1)$ we have $f(p_{i+1}, q_{i+1})\leq p_ir-c+\frac{c}{10}\log_2(4\Delta q_i)=f(p_i, q_i)-8c/10$. In case of outcome $(2)$ we have $f(p_{i+1}, q_{i+1})\leq p_ir+12c+\log_2(\Delta q_i)-200\cdot c/10\leq f(p_i, q_i)-8c$.

Noting that $f(p_0, q_0)\leq r(\ell+\frac{1}{16r})+O(\log r)$, we see that for some $i=O(\sqrt{r})$ steps we have $f(p_i, q_i)\leq r(\ell+\frac{1}{16r}) - 3\sqrt{r}\cdot 1/r\leq r(\ell-\frac{2}{\sqrt{r}})$. Hence, for this $i$, we either have $q_i\leq \frac{1}{10^4 r}$, in which case we obtain a half-sized constant submatrix and we are done, or we have $p_i\leq \ell-\frac{2}{\sqrt{r}}$, thus finding our matrix $M_s$ and stopping the process. This concludes the proof. 
\end{proof}


\begin{thebibliography}{10}

\bibitem{ADGS15}
N.~Alon, S.~Das, R.~Glebov, and B.~Sudakov.
\newblock Comparable pairs in families of sets.
\newblock {\em Journal of Combinatorial Theory, Series B}, 115:164--185, 2015.

\bibitem{AF85}
N.~Alon and P.~Frankl.
\newblock The maximum number of disjoint pairs in a family of subsets.
\newblock {\em Graphs Combin.}, 1(1):13--21, 1985.

\bibitem{AGG20}
N.~Alon, S.~Gilboa, and S.~Gueron.
\newblock A probabilistic variant of {S}perner's theorem and of maximal {$r$}-cover free families.
\newblock {\em Discrete Math.}, 343(10):112027, 4, 2020.

\bibitem{AlonSpencer}
N.~Alon and J.~H.~Spencer.
\newblock The Probabilistic Method. 
\newblock John Wiley \& Sons, 2016.

\bibitem{BSLRZ14}
E.~Ben-Sasson, S.~Lovett, and N.~Ron-Zewi.
\newblock An additive combinatorics approach relating rank to communication complexity.
\newblock {\em J. ACM}, 61(4), July 2014.

\bibitem{BSRZ15}
E.~Ben-Sasson and N.~Ron-Zewi.
\newblock From affine to two-source extractors via approximate duality.
\newblock {\em SIAM Journal on Computing}, 44(6):1670--1697, 2015.

\bibitem{BDL13}
A.~Bhowmick, Z.~Dvir, and S.~Lovett.
\newblock New bounds for matching vector families.
\newblock In {\em Proceedings of the Forty-Fifth Annual ACM Symposium on Theory of Computing}, pages 823--832, New York, NY, USA, 2013. Association for Computing Machinery.

\bibitem{chazelle}
B.~Chazelle.
\newblock Cutting hyperplanes for divide-and-conquer.
\newblock {\em Discrete \& Computational Geometry}, 9(2):145--158, 1993.

\bibitem{DR82}
A.~G. D`yachkov and V.~V. Rykov.
\newblock Bounds on the length of disjunctive codes.
\newblock {\em Problems Inform. Transmission}, 18(3):166--171, 1982.

\bibitem{EFF85}
P.~Erd\H{o}s, P.~Frankl, and Z.~F\"uredi.
\newblock Families of finite sets in which no set is covered by the union of {$r$} others.
\newblock {\em Israel J. Math.}, 51(1-2):79--89, 1985.

\bibitem{erdoskorado}
P.~Erd\H{o}s, C.~Ko, and R.~Rado.
\newblock Intersection theorems for systems of finite sets.
\newblock {\em Quart. J. Math. Oxford Ser.(2)}, 12:313--320, 1961.

\bibitem{FS09}
J. Fox and B. Sudakov, Density theorems for bipartite graphs and related Ramsey-type results, Combinatorica 29 (2009) 153-196.

\bibitem{F96}
Z.~F\"uredi.
\newblock On {$r$}-cover-free families.
\newblock {\em J. Combin. Theory Ser. A}, 73(1):172--173, 1996.

\bibitem{GGMT23}
W.~T. Gowers, B.~Green, F.~Manners, and T.~Tao.
\newblock On a conjecture of Marton.
\textit{preprint}, arXiv:2311.05762.

\bibitem{GGMT24}
W.~T. Gowers, B.~Green, F.~Manners, and T.~Tao.
\newblock Marton's conjecture in abelian groups with bounded torsion.
\textit{preprint}, arXiv:2404.02244.

\bibitem{Groth}
A.~Grothendieck.
\newblock {R{\'e}sum{\'e} de la th{\'e}orie m{\'e}trique des produits tensoriels topologiques}, volume~2.
\newblock {\em Soc. de Matem{\'a}tica de S{\~a}o Paulo}
, 1956.

\bibitem{Guy83}
R.~Guy.
\newblock Unsolved {P}roblems: {A} {M}iscellany of {E}rd\H{o}s {P}roblems.
\newblock {\em Amer. Math. Monthly}, 90(2):118+119--120, 1983.

\bibitem{HMST}
Z.~Hunter, A.~Milojevi\'c, B.~Sudakov, and I.~Tomon.
K\H{o}v\'ari-{S}\'os-{T}ur\'an theorem for hereditary families.
\textit{preprint}, arXiv:2401.10853.

\bibitem{KS64}
W.~Kautz and R.~Singleton.
\newblock Nonrandom binary superimposed codes.
\newblock {\em IEEE Transactions on Information Theory}, 10(4):363--377, 1964.

\bibitem{LS}
N.~Linial and A.~Shraibman.
\newblock Learning complexity vs communication complexity.
\newblock {\em Combinatorics, Probability and Computing}, 18(1-2):227--245, 2009.

\bibitem{LS88}
L.~Lov\'asz and M.~Saks.
\newblock Communication complexity and combinatorial lattice theory.
\newblock volume~47, pages 322--349. 1993.
\newblock 29th Annual IEEE Symposium on Foundations of Computer Science (White Plains, NY, 1988).

\bibitem{L16}
S.~Lovett.
\newblock Communication is bounded by root of rank.
\newblock {\em J. ACM}, 63(1):Art. 1, 9, 2016.

\bibitem{MST}
A. Milojevi\'c, B. Sudakov, I. Tomon.
\newblock Incidence bounds via extremal graph theory.
\newblock \textit{preprint}, arxiv:2401.06670.

\bibitem{Banff}
I.~Rival.
\newblock In {\em Ordered sets : proceedings of the NATO Advanced Study Institute held at Banff, Canada, August 28 to September 12, 1981}, NATO advanced study institutes series. Series C, Mathematical and physical sciences ; Volume 83, page 860, Dordrecht, The Netherlands ;, 1982. NATO Advanced Study Institute (1981 : Banff, Alta.), D. Reidel Publishing Company.

\bibitem{Sgall}
J.~Sgall.
\newblock Bounds on pairs of families with restricted intersections.
\newblock {\em Combinatorica}, 19(4):555--566, 1999.

\bibitem{SS22}
N.~Singer and M.~Sudan.
\newblock Point-hyperplane incidence geometry and the log-rank conjecture.
\newblock {\em ACM Trans. Comput. Theory}, 14(2):Art. 7, 16, 2022.

\bibitem{zara_survey}
S.~Smorodinsky.
\newblock A survey of {Z}arankiewicz problem in geometry.
\textit{preprint}, arXiv:2410.03702.

\bibitem{ST24}
B.~Sudakov and I.~Tomon.
\newblock Matrix discrepancy and the log-rank conjecture.
\newblock {\em Mathematical Programming}, 2024.

\end{thebibliography}
\end{document}